\documentclass{article}
\usepackage{euscript,amsmath,amssymb,amsfonts,graphicx,bm,amsthm,mathtools}

  \usepackage{paralist}
  \usepackage{graphics} 
\usepackage[caption=false]{subfig}
\usepackage{soul}

\usepackage{latexsym,amssymb,enumerate,amsmath,amsthm} 
  \usepackage{graphicx} 
  \usepackage{tikz}
     \usepackage[colorlinks=false]{hyperref}
 \hypersetup{urlcolor=blue, citecolor=red}
\usepackage{latexsym,amssymb,enumerate,amsmath} 
\usepackage{pgfplots}
\usepackage{soul,xcolor}

\newcommand{\dd}{\text{d}}
\def\eps{\varepsilon}
\def\E{\mathbb{E}}
\def\P{\mathbb{P}}
\def\R{\mathbb{R}}

\def\dist{\textup{d}}
\def\T{\mathcal{T}}

\def\D{\prescript{}{0}D_{t}^{1-\alpha}}
\def\L{\mathcal{L}_{\alpha}}
\def\O{\mathcal{O}}
\newcommand{\drie}{L_{\textup{rie}}}

\newtheorem{theorem}{Theorem}

\newtheorem{lemma}[theorem]{Lemma}
\newtheorem{corollary}[theorem]{Corollary}
\theoremstyle{plain}
\theoremstyle{remark}
\newtheorem{remark}[theorem]{Remark}
\theoremstyle{definition}
\newtheorem*{definition*}{Definition}

\begin{document}


\title{Extreme statistics of anomalous subdiffusion following a fractional Fokker-Planck equation: Subdiffusion is faster than normal diffusion}


\author{Sean D. Lawley\thanks{Department of Mathematics, University of Utah, Salt Lake City, UT 84112 USA (\texttt{lawley@math.utah.edu}).}
}
\date{\today}
\maketitle

\begin{abstract}
Anomalous subdiffusion characterizes transport in diverse physical systems and is especially prevalent inside biological cells. In cell biology, the prevailing model for chemical activation rates has recently changed from the first passage time (FPT) of a single searcher to the FPT of the fastest searcher out of many searchers to reach a target, which is called an extreme statistic or extreme FPT. In this paper, we investigate extreme statistics of searchers which move by anomalous subdiffusion. We model subdiffusion by a fractional Fokker-Planck equation involving the Riemann-Liouville fractional derivative. We prove an explicit and very general formula for every moment of subdiffusive extreme FPTs and approximate their full probability distribution. While the mean FPT of a single subdiffusive searcher is infinite, the fastest subdiffusive searcher out of many subdiffusive searchers typically has a finite mean FPT. In fact, we prove the counterintuitive result that extreme FPTs of subdiffusion are faster than extreme FPTs of normal diffusion. Mathematically, we employ a stochastic representation involving a random time change of a standard Ito drift-diffusion according to the trajectory of the first crossing time inverse of a Levy subordinator. A key step in our analysis is generalizing Varadhan's formula from large deviation theory to the case of subdiffusion, which yields the short-time distribution of subdiffusion in terms of a certain geodesic distance.
\end{abstract}

\section{Introduction}

Many complex systems are characterized by anomalous subdiffusion \cite{oliveira2019, klafter2005, hofling2013, barkai2012, sokolov2012, meroz2015}. The hallmark of subdiffusion is that the mean-squared displacement of a subdiffusive particle grows sublinearly in time. More precisely, if $X_{\alpha}(t)$ denotes the one-dimensional position of a subdiffusive particle at time $t\ge0$, then
\begin{align}\label{msd}
\E\big[\big(X_{\alpha}(t)-X_{\alpha}(0)\big)^{2}\big]
\,\propto\, t^{\alpha},\quad\alpha\in(0,1),
\end{align}
{ where $\E$ denotes expected value (normal diffusion corresponds to $\alpha=1$). In this paper, we model subdiffusion by a fractional Fokker-Planck equation (FPE) \cite{metzler1999}, but note that there are other models which yield \eqref{msd}, including fractional Brownian motion and generalized Langevin equations \cite{magdziarz2009, meroz2015, mckinley2018}.}

Subdiffusive dynamics have been found in diverse scenarios, including charge transport in amorphous semiconductors in photocopiers \cite{scher1975}, subsurface hydrology \cite{berkowitz2002}, and the movement of a bead in a polymer network \cite{amblard1996}. Subdiffusion is particularly prevalent in cell biology, where the phenomenon often stems from macromolecular crowding inside a cell \cite{golding2006}. The packing of organelles, proteins, lipids, sugars, and various filamentous networks in the cell can impede transport and make diffusion coefficients measured in dilute solution effectively meaningless \cite{hofling2013}. Indeed, crowded intracellular environments have been shown to significantly affect signaling pathways and search processes compared to diffusion in an empty medium \cite{isaacson2011, woringer2014, ma2020}.

An important quantity describing a randomly moving particle is the first time the particle (the ``searcher'') reaches some particular location (the ``target''), which is called a first passage time (FPT) \cite{redner2001}. In fact, FPTs determine the timescales in many physical, chemical, and biological systems \cite{redner2001}. Many theoretical and numerical studies focus on the FPT of a searcher that moves by normal diffusion \cite{benichou2008,ward10,ward10b,holcman2014,holcman2014time,grebenkov2017}, but less attention has been given to the FPT of a subdiffusive searcher \cite{lua2005,yuste2007,condamin2007,condamin2008,grebenkov2010}.

Recently, there has been significant interest and excitement in the literature regarding so-called extreme FPTs or fastest FPTs \cite{lawley2020esp, lawley2020uni, lawley2020dist, basnayake2019, schuss2019, coombs2019, redner2019, sokolov2019, rusakov2019, martyushev2019, tamm2019, basnayake2019c,godec2016x,hartich2018,hartich2019}. Mathematically, an extreme FPT is defined by
\begin{align}\label{tn}
T_{N}
:=\min\{\tau_{1},\dots,\tau_{N}\},
\end{align}
where $\{\tau_{1},\dots,\tau_{N}\}$ are $N\ge1$ independent and identically distributed (iid) FPTs. More generally, the $k$th fastest FPT is
\begin{align}\label{tkn}
T_{k,N}
:=\min\big\{\{\tau_{1},\dots,\tau_{N}\}\backslash\cup_{j=1}^{k-1}\{T_{j,N}\}\big\},\quad k\in\{1,\dots,N\},
\end{align}
where $T_{1,N}:=T_{N}$. The interest in extreme FPTs stems from the fact that many processes involve a large collection of simultaneous searchers in which the first searcher to find the target triggers an event. For example, gene regulation depends on only the fastest few transcription factors to reach a specific gene location out of roughly $N\in[10^{2},10^{4}]$  transcription factors \cite{harbison2004, godec2016x}. Similarly, human fertilization depends on the fastest sperm cell to find the egg out of roughly $N=10^{8}$ sperm cells \cite{meerson2015}.

In this paper, we investigate extreme FPTs of subdiffusive searchers. We consider subdiffusive searchers whose probability density satisfies a fractional { FPE}. Fractional FPEs were introduced in \cite{metzler1999} and generalize fractional diffusion equations \cite{schneider1989}. Fractional FPEs describe non-Markovian processes and include trapping phenomena through the Riemann-Liouville fractional differential operator \cite{samko1993}. Our approach relies on a certain stochastic representation of the subdiffusive paths  corresponding to fractional FPEs. This stochastic representation is a random time change of an It\^{o} drift-diffusion according to the trajectory of the first crossing time inverse of an independent L\'{e}vy subordinator (see \cite{magdziarz2016} and the references therein).

Our analysis yields an explicit formula for the leading order behavior of every moment of the $k$th fastest subdiffusive FPT, $T_{k,N}$, as the number of searchers grows. We note that while the mean FPT, $\E[\tau]$, of a single subdiffusive searcher is typically infinite \cite{yuste2004}, the mean of $T_{k,N}$ for subdiffusion is typically finite for large $N$ (as we prove below). In particular, we prove that for any moment $m\ge1$ and any $k\ge1$,
\begin{align}\label{leading}
\E[(T_{k,N})^{m}]
\sim\bigg(\frac{t_{\alpha}}{(\ln N)^{2/\alpha-1}}\bigg)^{m}
\quad\text{as }N\to\infty,
\end{align}
where $t_{\alpha}>0$ is the characteristic subdiffusive (or diffusive) timescale,
\begin{align*}
t_{\alpha}
:=\Big(\alpha^{\alpha}(2-\alpha)^{2-\alpha}\frac{L^{2}}{4K_{\alpha}}\Big)^{1/\alpha}>0,\quad \alpha\in(0,1],
\end{align*}
where $\alpha\in(0,1]$ is the mean-squared displacement exponent (as in \eqref{msd}), $K_{\alpha}>0$ is the generalized diffusion coefficient { (with dimension $(\text{length})^{2}(\text{time})^{-\alpha}$)}, and $L>0$ is a certain geodesic distance between the possible searcher starting locations and the target (given precisely below). { Throughout this paper, ``$f\sim g$'' means $f/g\to1$ in the limit indicated (which is as $N\to\infty$ in \eqref{leading})}. The formula~\eqref{leading} holds in significant generality, including subdiffusion in $\R^{d}$ with general space-dependent drift and diffusion coefficients. Since it was recently proven that \eqref{leading} holds for normal diffusion with $\alpha=1$ \cite{lawley2020uni}, we obtain the counterintuitive result that extreme subdiffusive FPTs are faster than extreme diffusive FPTs.

Further, assuming the short-time behavior of the survival probability of a single FPT for normal diffusion is known, we obtain the limiting probability distribution of the $k$th fastest FPT of subdiffusion and explicit three-term (or higher) asymptotic expansions for every moment as $N\to\infty$. This  probability distribution is described in terms of the classical Gumbel distribution.

A key step in our analysis is generalizing Varadhan's formula to subdiffusion. Varadhan's formula is a fundamental result in large deviation theory which determines the short-time behavior of the probability density of a drift-diffusion process on a logarithmic scale in terms of a certain geodesic distance \cite{varadhan1967}. By generalizing this result to subdiffusive processes, we obtain the short-time behavior of (i) the probability density for the position and (ii) the survival probability for the FPT for any subdiffusive process for which the short-time behavior of the corresponding diffusive process is known. This result allows us to employ methods that were developed for studying extreme FPTs of normal diffusion \cite{lawley2020uni,lawley2020dist}. 

The rest of the paper is organized as follows. In section~\ref{prelim}, we introduce notation and recall some facts about subdiffusion and fractional FPEs. In section~\ref{varadhan}, we generalize Varadhan's formula to the case of subdiffusion. In section~\ref{extreme}, we analyze extreme FPTs of subdiffusion. In section~\ref{examples}, we apply our results to several examples. We conclude by discussing relations to previous work. An Appendix collects several proofs.

\section{Preliminaries}\label{prelim}

We begin by introducing notation and recalling several results about subdiffusion modeled by a fractional FPE. Let $\{X_{\alpha}(t)\}_{t\ge0}$ be the position of a $d$-dimensional subdiffusive searcher with $d\ge1$. Let $p_{\alpha}(x,t\,|\,x_{0},0)$ be the probability density that $X_{\alpha}(t)=x\in\R^{d}$ given $X_{\alpha}(0)=x_{0}\in\R^{d}$. That is,
\begin{align*}
p_{\alpha}(x,t\,|\,x_{0},0)\,\dd x
=\P(X_{\alpha}(t)=x\,|\,X_{\alpha}(0)=x_{0}).
\end{align*}
We are interested in the case that the probability density $p_{\alpha}$ satisfies the fractional FPE,
\begin{align}\label{ffpe}
\begin{split}
\frac{\partial}{\partial t}p_{\alpha}
&=\L\D p_{\alpha},\quad x\in\R^{d},\,t>0,\\
p_{\alpha}
&=\delta(x-x_{0}),\quad x\in\R^{d},\,t=0.
\end{split}
\end{align}
Here, $\alpha\in(0,1)$ and $\D$ is the fractional derivative of Riemann-Liouville type \cite{samko1993}, defined by
\begin{align*}
\D f(t)
=\frac{1}{\Gamma(\alpha)}\frac{\dd}{\dd t}\int_{0}^{t}\frac{f(s)}{(t-s)^{1-\alpha}}\,\dd s,
\end{align*}
where $\Gamma(\alpha)=\int_{0}^{\infty}u^{\alpha-1}e^{-u}\,\dd u$  is the Gamma function. Further, $\L$ is the forward Fokker-Planck operator,
\begin{align}\label{lop}
\L f(x)
:=-\sum_{i=1}^{d}\frac{\partial}{\partial x_{i}}\bigg[\frac{b_{i}(x)}{\eta_{\alpha}}f(x)\bigg]
+K_{\alpha}\sum_{i=1}^{d}\sum_{j=1}^{d}\frac{\partial^{2}}{\partial x_{i}\partial x_{j}}
\Big[\big({{\Sigma}}(x){{\Sigma}}(x)^{\top}\big)_{i,j}f(x)\Big],
\end{align}
where $\eta_{\alpha}>0$ is a generalized friction coefficient with dimension $(\text{time})^{\alpha-2}$ \cite{metzler1999}, $b(x):\R^{d}\mapsto\R^{d}$ is a space-dependent vector with dimension $(\text{length})(\text{time})^{-2}$ describing the drift, $K_{\alpha}>0$ is a generalized diffusion coefficient with dimension $(\text{length})^{2}(\text{time})^{-\alpha}$, and ${{\Sigma}}(x):\R^{d}\mapsto\R^{d\times m}$ is a dimensionless function describing any space dependence or anisotropy in the diffusion. Assume $b$ and ${{\Sigma}}$
satisfy mild conditions (namely that $b$ is uniformly bounded
and uniformly Lipschitz continuous and that ${{\Sigma}}{{\Sigma}}^{\top}$ is uniformly Lipschitz continuous and its eigenvalues are bounded above $\gamma_{1}>0$ and bounded below $\gamma_{2}>\gamma_{1}$). Note that $\L$ and $\D$ commute since $b$ and ${{\Sigma}}$ and independent of time.

It is well-known that a subdiffusive process whose probability density satisfies a fractional FPE can be written as a random time change of a diffusive process satisfying an It\^{o} stochastic differential equation (SDE) \cite{magdziarz2016,umarov2016,magdziarz2007,meerschaert2002}. Specifically, throughout this paper we let $\{U_{\alpha}(s)\}_{s\ge0}$ be an $\alpha$-stable subordinator \cite{janicki1993, sato1999} with Laplace transform
\begin{align}\label{laplacetransform}
\E[e^{-rU_{\alpha}(s)}]=e^{-sr^{\alpha}},\quad \alpha\in(0,1).
\end{align}
Let $S_{\alpha}(t)$ be the inverse $\alpha$-stable subordinator,
\begin{align}\label{S}
S_{\alpha}(t)
:=\inf\{s>0:U_{\alpha}(s)>t\}.
\end{align}
Sample paths of $U_{\alpha}$ are continuous from the right with left hand limits, strictly increasing, and satisfy $U_{\alpha}(0)=0$ and $U_{\alpha}(s)\to\infty$ as $s\to\infty$. Therefore, the inverse process \eqref{S} is well-defined and has almost surely continuous sample paths.

Let $\{X_{1}(s)\}_{s\ge0}$ be a $d$-dimensional diffusion process satisfying the It\^{o} SDE,
\begin{align}\label{sde}
\dd X_{1}(s)
=\frac{b(X_{1})}{\eta_{\alpha}}\,\dd s+\sqrt{2K_{\alpha}}{{\Sigma}}(X_{1})\,\dd W(s),\quad
X_{1}(0)=x_{0}\in\R^{d},
\end{align}
where $\{W(s)\}_{s\ge0}\in\R^{m}$ is a standard Brownian motion independent of $U_{\alpha}$. Note that $\L$ in \eqref{lop} is the forward Fokker-Planck operator corresponding to \eqref{sde}. Further note that $\{U_{\alpha}(s)\}_{s\ge0}$ and $\{X_{1}(s)\}_{s\ge0}$ are indexed by the ``internal time'' $s\ge0$, which is not real, physical time, and in fact has dimension $(\text{time})^{\alpha}$.

We then construct the subdiffusive process $X_{\alpha}$ from the diffusive process $X_{1}$ indexed by $S_{\alpha}$. Specifically, we define
\begin{align}\label{timechange}
X_{\alpha}(t)
:=X_{1}(S_{\alpha}(t)),\quad t\ge0.
\end{align}
It is well-known that the probability density of \eqref{timechange} satisfies the fractional FPE in \eqref{ffpe} (see Theorem 2.1 in \cite{carnaffan2017} for a proof in this particular multidimensional setting). Note that $X_{\alpha}$ has continuous sample paths since $S_{\alpha}$ and $X_{1}$ have continuous sample paths (with probability one). See Figure~\ref{figtimechange} for an illustration.

\begin{figure}[t]
\centering
\includegraphics[width=1\linewidth]{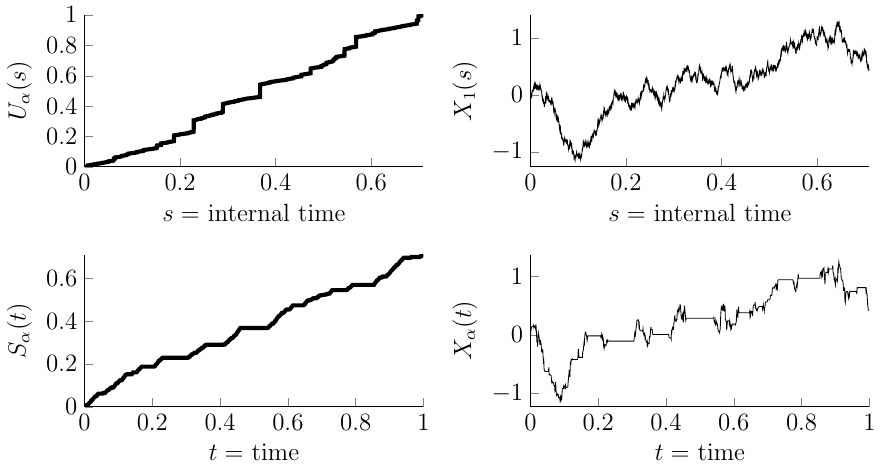}
\caption{Construction of a subdiffusive process as a random time change of a diffusive process. The top left panel plots a realization of the $\alpha$-stable subordinator $U_{\alpha}(s)$ as a function of the internal time $s$. The bottom left panel plots the corresponding realization of the inverse $\alpha$-stable subordinator $S_{\alpha}(t)$ as a function of the physical time $t$. Notice that jumps of $U_{\alpha}$ correspond to time intervals in which $S_{\alpha}$ is constant (i.e.\ pauses in $S_{\alpha}$). To top right panel plots the normal diffusive process $X_{1}(s)$. The bottom right panel plots the subdiffusive process $X_{\alpha}(t)=X_{1}(S_{\alpha}(t))$. Notice that (i) $X_{\alpha}$ pauses when $S_{\alpha}$ pauses and (ii) the path of $X_{\alpha}$ otherwise resembles $X_{1}$. This simulation employs the method in \cite{magdziarz2007}.}
\label{figtimechange}
\end{figure}

The construction of $X_{\alpha}$ in \eqref{timechange} means that we can study $X_{\alpha}$ by studying $X_{1}$ and $S_{\alpha}$. Since the driving Brownian motion $W$ in \eqref{sde} and the subordinator $U_{\alpha}$ are independent, the diffusion $X_{1}$ and the inverse subordinator $S_{\alpha}$ are independent. Therefore, conditioning on the value of $S_{\alpha}(t)$ yields that the probability density of $X_{\alpha}$ satisfying the fractional FPE \eqref{ffpe} is given by
\begin{align}\label{prep}
p_{\alpha}(x,t\,|\,x_{0},0)
=\int_{0}^{\infty}q_{\alpha}(s,t)p_{1}(x,s\,|\,x_{0},0)\,\dd s,
\end{align}
where $q_{\alpha}(s,t)$ is the probability density that $S_{\alpha}(t)=s$ and $p_{1}(x,s\,|\,x_{0},0)$ is the probability density that $X_{1}(s)=x$ (which satisfies \eqref{ffpe} without the Riemann-Liouville operator $\D$). Furthermore, the definition of $S_{\alpha}(t)$ in \eqref{S} and the self-similarity property that {$\{c^{-1/\alpha}U_{\alpha}(cs)\}_{s\ge0}$} is equal in distribution to {$\{U_{\alpha}(s)\}_{s\ge0}$} for all $c>0$ implies that
\begin{align}\label{rela}
\P(S_{\alpha}(t)\le s)
=\P(t\le U_{\alpha}(s))
=\P(t\le s^{1/\alpha}U_{\alpha}(1)).
\end{align}
Therefore, differentiating \eqref{rela} yields the probability density 
\begin{align*}
q_{\alpha}(s,t)
:=\frac{\dd}{\dd s}\P(S_{\alpha}(t)\le s)
=\frac{t}{\alpha s^{1+1/\alpha}}l_{\alpha}\Big(\frac{t}{s^{1/\alpha}}\Big),
\end{align*}
where $l_{\alpha}(z)$ is the probability density of {$U_{\alpha}(1)$}.

The probability density $l_{\alpha}(z)$ is usually defined by its Laplace transform,
\begin{align}\label{lltdef}
\int_{0}^{\infty}e^{-rz}l_{\alpha}(z)\,\dd z
=e^{-r^{\alpha}},
\end{align}
since a simple explicit formula for $l_{\alpha}(z)$ for arbitrary $\alpha\in(0,1)$ is unknown {\cite{penson2010, kosztolowicz2004}}. The density $l_{\alpha}(z)$ has the small $z$ behavior \cite{schneider1986, barkai2001},
\begin{align}\label{l}
l_{\alpha}(z)
&\sim Bz^{-{\xi}}e^{-\kappa/z^{{\theta}}}\quad\text{as }z\to0+,
\end{align}
where
\begin{align}\label{ldef}
\begin{split}
{\theta}
&=\frac{\alpha}{1-\alpha},\quad
\kappa
=(1-\alpha)\alpha^{\alpha/(1-\alpha)},\\
B
&=\sqrt{\frac{\alpha^{1/(1-\alpha)}}{2\pi(1-\alpha)}},\quad
{\xi}
=\frac{2-\alpha}{2(1-\alpha)}.
\end{split}
\end{align}

We can similarly express the distribution of FPTs of the subdiffusive process $X_{\alpha}$ in terms of $q_{\alpha}$ and the distribution of FPTs of the diffusive process $X_{1}$. Let ${\sigma}$ be the FPT of $X_{1}$ to some target set $\Omega_{\text{T}}\subset\R^{d}$,
\begin{align*}
{\sigma}
:=\inf\{s>0:X_{1}(s)\in \Omega_{\text{T}}\}.
\end{align*}
It follows that the FPT of $X_{\alpha}$ to the target is given by $U_{\alpha}({\sigma})$,
\begin{align}\label{relatime}
{\tau}
:=\inf\{t>0:X_{\alpha}(t)\in \Omega_{\text{T}}\}
=U_{\alpha}({\sigma}).
\end{align}
Therefore, the distribution of ${\tau}$ is obtained by integrating $q_{\alpha}$ against the distribution of ${\sigma}$,
\begin{align}\label{taurep}
\P({\tau}\le t)
=\int_{0}^{\infty}q_{\alpha}(s,t)\P({\sigma}\le s)\,\dd s.
\end{align}

\section{Short-time distributions for subdiffusion}\label{varadhan}

In this section, we use the representations \eqref{prep} and \eqref{taurep} to study the short-time behavior of the subdiffusive probability density $p_{\alpha}$ and the distribution of the subdiffusive FPT ${\tau}$. In particular, we use the asymptotics of $l_{\alpha}(z)$ in \eqref{l} and the fact that diffusive probability densities and diffusive FPTs have the following short-time behavior under very general conditions \cite{varadhan1967,lawley2020uni},
\begin{align}
\lim_{s\to0+}s\ln p_{1}(x,s\,|\,x_{0},0)
&=-C_{1}'<0,\label{vvv}\\
\lim_{s\to0+}s\ln \P({\sigma}\le s)
&=-C_{1}<0,\label{sss}
\end{align}
for constants $C_{1}'>0,C_{1}>0$. Equation~\eqref{vvv} is known as Varadhan's formula \cite{varadhan1967} and generally holds as long as $x\neq x_{0}$. Equation~\eqref{sss} generally holds as long as the diffusive searcher $X_{1}$ cannot start arbitrarily close to the target \cite{lawley2020uni}. Furthermore, in many scenarios it is possible to obtain more detailed information than \eqref{vvv}-\eqref{sss}, and in fact to show that
\begin{align}
p_{1}(x,s\,|\,x_{0},0)
&\sim A_{1}'s^{p_{1}'}e^{-C_{1}'/s}\quad\text{as }s\to0+,\label{vvv2}\\
\P({\sigma}\le s)
&\sim A_{1}s^{p_{1}}e^{-C_{1}/s}\quad\text{as }s\to0+,\label{sss2}
\end{align}
for constants $A_{1}'>0$, $A_{1}>0$, $p_{1}'\in\R$, $p_{1}\in\R$, and $C_{1}'>0$, $C_{1}>0$. { For example, see the scenarios in sections~\ref{half}-\ref{narrow} where we show that \eqref{sss2} holds.} Throughout this paper,
\begin{align*}
``f\sim g''\quad\text{means}\quad
f/g\to1,
\end{align*}
{ in the limit indicated (which is in the limit $s\to0+$ in \eqref{vvv2}-\eqref{sss2}).}

\subsection{General integral asymptotics}

We start with two general results on the asymptotic behavior of integrals of the form \eqref{prep} and \eqref{taurep} assuming the functions in the integrands have the asymptotic behaviors in either \eqref{vvv}-\eqref{sss} or \eqref{vvv2}-\eqref{sss2}. We note that in Theorems~\ref{log} and \ref{more} below, the function $F_{1}(s)$ plays the role of either $p_{1}(x,s\,|x_{0},0)$ or $\P(\sigma\le s)$ in \eqref{vvv}-\eqref{sss2}, and $l(z)$ plays the role of $l_{\alpha}(z)$ in \eqref{lltdef}, but the theorems are stated and proven for general functions $F_{1}(s)$ and $l(z)$.

\begin{theorem}\label{log}
Assume $F_{1}(s)$ is a bounded, positive function that satisfies
\begin{align}\label{logF}
\lim_{s\to0+}s\ln F_{1}(s)=-C_{1}<0,
\end{align}
for some constant $C_{1}>0$. Assume $l(z)$ is a positive function that satisfies $\int_{0}^{\infty}l(z)\,\dd z<\infty$ and
\begin{align}\label{logl}
\lim_{z\to0+}z^{{\theta}}\ln l(z)
=-\kappa<0,
\end{align}
for some constants ${\theta}>0$ and $\kappa>0$. If $\alpha>0$ and 
\begin{align*}
F_{\alpha}(t)
:=\int_{0}^{\infty}\frac{t}{\alpha s^{1+1/\alpha}}
l\Big(\frac{t}{s^{1/\alpha}}\Big)F_{1}(s)\,\dd s,
\end{align*}
then
\begin{align*}
\lim_{t\to0+}t^{\beta}\ln F_{\alpha}(t)=-C<0,
\end{align*}
where
\begin{align}\label{bcdef}
\beta
=\frac{\alpha{\theta}}{\alpha+{\theta}},\quad
C
:=C_{1}\Big(\frac{\kappa{\theta}}{C_{1}\alpha}\Big)^{\frac{\alpha}{\alpha+{\theta}}}\Big(\frac{\alpha+{\theta}}{{\theta}}\Big).
\end{align}
\end{theorem}

Theorem~\ref{log} gives the logarithmic asymptotic behavior of the integral $F_{\alpha}(t)$ assuming that only the logarithmic asymptotic behavior of $l(z)$ and $F_{1}(s)$ are known. The next theorem assumes more detailed information about the asymptotic behavior of $l(z)$ and $F_{1}(s)$ to obtain more detailed information about the asymptotic behavior of $F_{\alpha}(t)$.

\begin{theorem}\label{more}
Under the assumptions of Theorem~\ref{log}, assume further that
\begin{align}
F_{1}(s)
&\sim A_{1}s^{p_{1}}e^{-C_{1}/s}\quad\text{as }s\to0+,\label{asympF}\\
l(z)
&\sim Bz^{-{\xi}}e^{-\kappa/z^{{\theta}}}\quad\text{as }z\to0+,\label{asympl}
\end{align}
where $A_{1}>0$, $p_{1}\in\R$, $C_{1}>0$, $B>0$, ${\xi}\in\R$, $\kappa>0$, ${\theta}>0$. Then
\begin{align*}
F_{\alpha}(t)
:=\int_{0}^{\infty}\frac{t}{\alpha s^{1+1/\alpha}}
l\Big(\frac{t}{s^{1/\alpha}}\Big)F_{1}(s)\,\dd s
&\sim At^{p}e^{-C/t^{\beta}}\quad\text{as }t\to0+,
\end{align*}
where $\beta$ and $C$ are given in \eqref{bcdef} and
\begin{align}\label{complicated2}
A
&:= A_{1} B \sqrt{\frac{2\pi}{C_{1}\alpha(\alpha +{\theta} )}} \left(\frac{\alpha  C_{1}}{\kappa  {\theta} }\right)^{\frac{\alpha +2 \alpha  p_{1}+2 {\xi} -2}{2 (\alpha +{\theta} )}},\quad
p
:=\frac{\alpha  (2 p_{1} {\theta} -2 {\xi} +{\theta} +2)}{2 (\alpha +{\theta} )}.
\end{align}
\end{theorem}

\subsection{Application to FPTs}

We now show how Theorems~\ref{log} and \ref{more} can be used to study the FPT distribution of a subdiffusive process. In particular, if the short-time behavior of the distribution of a diffusive FPT is known, then Corollary~\ref{corlog} below gives the short-time behavior of the distribution of the corresponding subdiffusive FPT.

\begin{corollary}\label{corlog}[Subdiffusive FPT]
Let ${\sigma}$ be a random variable satisfying
\begin{align*}
\lim_{s\to0+}s\ln\P({\sigma}\le s)=-C_{1}<0,
\end{align*}
for some constant $C_{1}>0$. If $\{U_{\alpha}(t)\}_{t\ge0}$ is an $\alpha$-stable subordinator as in \eqref{laplacetransform} and is independent of ${\sigma}$, then
\begin{align*}
\lim_{t\to0+}t^{\beta}\ln\P(U_{\alpha}({\sigma})\le t)=-C<0,
\end{align*}
where
\begin{align}\label{baca}
\beta
&:=\frac{\alpha}{2-\alpha},\quad
C
:=(2-\alpha) \alpha ^{\frac{\alpha }{2-\alpha}} C_{1}^{\frac{1}{2-\alpha }},
\end{align}

If we additionally assume that
\begin{align*}
\P({\sigma}\le s)
\sim A_{1}s^{p_{1}}e^{-C_{1}/s}\quad\text{as }s\to0+,
\end{align*}
for some $A_{1}>0$, $p_{1}\in\R$, and $C_{1}>0$, then
\begin{align*}
\P(U_{\alpha}({\sigma})\le t)
&\sim At^{p}e^{-C/t^{\beta}}\quad\text{as }t\to0+,
\end{align*}
where $\beta$ and $C$ are in \eqref{baca} and 
\begin{align}\label{ap}
A
&:=\frac{A_{1} \Big(\alpha ^{\frac{-\alpha }{2-\alpha}} C_{1}^{\frac{1-\alpha}{2-\alpha}}\Big)^{p_{1}}}{\sqrt{\alpha (2-\alpha)}},\quad
p
:=\beta p_{1}.
\end{align}
\end{corollary}

\subsection{Varadhan's formula for subdiffusion}

We now show how Theorems~\ref{log} and \ref{more} can be used to study the probability density for the position of a subdiffusive process. Similar to the subsection above, if the short-time behavior of the probability density for the position of a diffusive searcher is known, then Corollaries~\ref{corv} and \ref{corg} below give the short-time behavior of the probability density for the position of the corresponding subdiffusive searcher. These results extend Varadhan's formula to subdiffusion \cite{varadhan1967}. 

We begin by reviewing Varadhan's formula. Let $p_{1}(x,s\,|\,x_{0},0)$ denote the probability density of the It\^{o} process $\{X_{1}(s)\}_{s\ge0}$ satisfying \eqref{sde}. For any smooth path $\omega:[0,1]\to \R^{d}$, we define the length of the path in the Riemannian metric corresponding to the inverse of the diffusion matrix $a:={{\Sigma}}{{\Sigma}}^{\top}$ in \eqref{sde},
\begin{align}\label{ll}
d(\omega)
:=\int_{0}^{1}\sqrt{\dot{\omega}^{\top}(s)a^{-1}(\omega(s))\dot{\omega}(s)}\,\dd s.
\end{align}
Then, Varadhan's formula is \cite{varadhan1967}
\begin{align}\label{vv}
\lim_{s\to0+}s\ln p_{1}(x,s\,|\,x_{0},0)
=-\frac{\drie^{2}({x_{0}},{{x}})}{4K_{\alpha}},
\end{align}
where $\drie$ is the geodesic length
\begin{align}\label{drie}
\begin{split}
\drie({x_{0}},{{x}})
&:=\inf\{d(\omega):\omega(0)=x_{0},\,\omega(1)=x\},
\end{split}
\end{align}
where the infimum in \eqref{drie} is over smooth paths $\omega:[0,1]\to \R^{d}$ which go from $\omega(0)={x_{0}}$ to $\omega(1)={{x}}$. In words, $\drie(x_{0},x)$ is the length of the optimal path from $x_{0}$ to $x$, where ``optimal'' balances being short in Euclidean length and avoiding areas of slow diffusion. Notice that $\drie$ is merely the Euclidean length of the straight line path from $x_{0}$ to $x$ if $a$ is the identity matrix, which corresponds to isotropic, spatially constant diffusion. Note also that $\drie$ is independent of the drift in \eqref{sde}.

\begin{corollary}\label{corv}[Varadhan's formula for subdiffusion]
Assume $p_{\alpha}(x,t\,|\,x_{0},0)$ satisfies the fractional FPE \eqref{ffpe}. If $x\neq x_{0}$, then
\begin{align*}
\lim_{t\to0+}t^{\beta}\ln p_{\alpha}(x,t\,|\,x_{0},0)=-C<0,
\end{align*}
where
\begin{align*}
\beta
&:=\frac{\alpha}{2-\alpha},
\quad
C
:=(2-\alpha) \alpha ^{\frac{\alpha }{2-\alpha}} \Big(\frac{\drie^{2}(x_{0},x)}{4K_{\alpha}}\Big)^{\frac{1}{2-\alpha }}.
\end{align*}
\end{corollary}


While we stated Varadhan's formula in \eqref{vv} for the case of diffusion in $\R^{d}$ with space-dependent drift and diffusivities, Varadhan's formula for diffusion has been extended to much more general scenarios (see, for example, \cite{norris1997}), including scenarios in which fractional FPEs have not been studied. Nevertheless, given a stochastic process $\{X_{1}(s)\}_{s\ge0}$, we can always define $X_{\alpha}(t):=X_{1}(S_{\alpha}(t))$. Hence, if the probability density of $X_{1}$ satisfies Varadhan's formula, then our analysis above yields the short-time behavior of the probability density of $X_{\alpha}$. The following corollary summarizes this point.

\begin{corollary}\label{corg}
Let $\{X_{1}(s)\}_{s\ge0}$ be a stochastic process and assume its probability density $p_{1}(x,s\,|\,x_{0},0)$ is a bounded function of $s\ge0$ and satisfies
\begin{align*}
\lim_{s\to0+}s\ln p_{1}(x,s\,|\,x_{0},0)=-C_{1}<0,
\end{align*}
for some constant $C>0$. Let $\{U_{\alpha}(t)\}_{t\ge0}$ be an $\alpha$-stable subordinator as in \eqref{laplacetransform} with inverse $\{S_{\alpha}(t)\}_{t\ge0}$ in \eqref{S}. Then, the probability density of $X_{\alpha}(t):=X_{1}(S_{\alpha}(t))$ satisfies
\begin{align*}
\lim_{t\to0+}t^{\beta}\ln p_{\alpha}(x,t\,|\,x_{0},0)=-C<0,
\end{align*}
where $\beta$ and $C$ are in \eqref{baca}.

If we additionally assume that
\begin{align*}
p_{1}(x,s\,|\,x_{0},0)
\sim A_{1}s^{p_{1}}e^{-C_{1}/s}\quad s\to0+,
\end{align*}
for some $A_{1}>0$, $p_{1}\in\R$, and $C_{1}>0$, then
\begin{align*}
p_{\alpha}(x,t\,|\,x_{0},0)
\sim At^{p}e^{-C/t^{\beta}}\quad t\to0+,
\end{align*}
where $\beta$, $C$, $A$, and $p$ are in \eqref{baca} and \eqref{ap}.
\end{corollary}

\section{Extreme first passage times of subdiffusion}\label{extreme}

In the previous section, we determined the typical short-time behavior of the distribution of FPTs of single subdiffusive searchers. In this section, we use these short-time behaviors to study extreme FPTs of $N\gg1$ subdiffusive searchers. The results in this section generalize results for extreme FPTs of normal diffusion proven in \cite{lawley2020uni,lawley2020dist}.

\subsection{Leading order moments}

We begin by finding the leading order behavior of the $m$th moment of the $k$th fastest subdiffusive FPT, $T_{k,N}$, defined in \eqref{tkn}. This leading order moment formula only requires the logarithmic short-time behavior of a single subdiffusive FPT (which, by Corollary~\ref{corlog}, only requires the logarithmic short-time behavior of a single diffusive FPT).

\begin{theorem}\label{uni}
Let $\{\tau_{n}\}_{n=1}^{\infty}$ be a sequence of iid realizations of any random variable $\tau$ satisfying
\begin{align}\label{conditionb}
\lim_{t\to0+}t^{\beta}\ln\P(\tau\le t)=-C<0,
\end{align}
for some constants $\beta\in(0,1]$ and $C>0$. Define $T_{N}$ and $T_{k,N}$ as in \eqref{tn}-\eqref{tkn} and assume that
\begin{align}\label{conditiona}
\E[T_{N}]
<\infty\quad\text{for some $N\ge1$}.
\end{align}
Then for any $m\ge1$ and $k\ge1$, we have that
\begin{align}\label{res}
\E[(T_{k,N})^{m}]
\sim\Big(\frac{C}{\ln N}\Big)^{m/\beta}\quad\text{as }N\to\infty.
\end{align}
\end{theorem}

Theorem~\ref{uni} describes how extreme subdiffusive FPTs vanish as $N\to\infty$. As an immediate corollary, we obtain that extreme FPTs are much less variable than single FPTs.

\begin{corollary}\label{corcv}
Under the assumptions of Theorem~\ref{log}, the variance of $T_{k,N}$,
\begin{align*}
\textup{Variance}[T_{k,N}]
:=\E\Big[\big(T_{k,N}-\E[T_{k,N}]\big)^{2}\Big],
\end{align*}
vanishes faster than $(\ln N)^{-2/\beta}$ as $N\to\infty$. That is,
\begin{align*}
(\ln N)^{2/\beta}\textup{Variance}[T_{k,N}]
\to0\quad\text{as }N\to\infty.
\end{align*}
Furthermore, the coefficient of variation of $T_{k,N}$, vanishes as $N\to\infty$. That is,
\begin{align*}
\frac{\sqrt{\textup{Variance}[T_{k,N}]}}{\E[T_{k,N}]}
\to0\quad\text{as }N\to\infty.
\end{align*}
\end{corollary}

Theorem~\ref{uni} assumes that $\E[{T_{N}}]<\infty$ for some $N\ge1$. While it is typically the case that a single subdiffusive FPT $\tau$ has infinite mean \cite{yuste2004}, the next result shows that $\E[{T_{N}}]<\infty$ for sufficiently large $N$ as long as the survival probability of the corresponding normal diffusive FPT vanishes no slower than algebraically at large time. In this paper, the notation
\begin{align*}
f(t)
=\O(g(t))\quad\text{as }t\to\infty,
\end{align*}
means that there exists $M>0$ and $t_{0}>0$ so that
\begin{align*}
|f(t)|
\le Mg(t)\quad\text{for all }t\ge t_{0}.
\end{align*}

\begin{theorem}\label{nice}
Let ${\sigma}\ge0$ be a nonnegative random variable and assume that there exists an $r>0$ so that
\begin{align}
\P({\sigma}>s)
=\O(s^{-r})\quad\text{as }s\to\infty.\label{asn}
\end{align}
Let $\{U_{\alpha}(t)\}_{t\ge0}$ be an $\alpha$-stable subordinator as in \eqref{laplacetransform}. If $U_{\alpha}$ and ${\sigma}$ are independent, then
\begin{align}\label{fr72}
\P(U_{\alpha}({\sigma})>t)
=\O(t^{-\alpha r/(1+r)})\quad\text{as }t\to\infty.
\end{align}
Furthermore, if ${T_{N}}$ is the minimum of $N\ge1$ iid realizations of $U_{\alpha}({\sigma})$, then
\begin{align}\label{fr73}
\E[{T_{N}}]
<\infty\quad\text{for all }N>\frac{1+r}{\alpha r}.
\end{align}
\end{theorem}

We note that for diffusive searchers in a bounded domain, the survival probability of the FPT ${\sigma}$ typically vanishes exponentially,
\begin{align}\label{asnexp}
\P({\sigma}>s)=\O(e^{-\lambda s})\quad\text{as }s\to\infty,
\end{align}
for some $\lambda>0$. Therefore, if \eqref{asnexp} is satisfied, then \eqref{asn} holds for every $r>0$ and Theorem~\ref{nice} implies
\begin{align*}
\E[{T_{N}}]
<\infty\quad\text{for all }N>1/\alpha.
\end{align*}


\subsection{Full distribution}\label{distribution}

In the previous subsection, we found the leading order behavior of the $m$th moment of the $k$th fastest subdiffusive FPT using only the logarithmic short-time behavior of the distribution of a single FPT. In this section, we show that we can find the full limiting distribution of the fastest subdiffusive FPT if we have more detailed information about the short-time distribution of a single subdiffusive FPT (which, by Corollary~\ref{corlog}, only requires the short-time behavior of a single diffusive FPT).

We begin by recalling the definition of the Gumbel distribution.
\begin{definition*}
A random variable $X$ has a \emph{Gumbel distribution} with location parameter $b\in\R$ and scale parameter $a>0$ if\footnote{A Gumbel distribution is sometimes defined differently by saying that if \eqref{xgumbel} holds, then $-X$ has a Gumbel distribution with shape $-b$ and scale $a$.}
\begin{align}\label{xgumbel}
\P(X>x)
=\exp\Big[-\exp\Big(\frac{x-b}{a}\Big)\Big],\quad\text{for all } x\in\R,
\end{align}
in which case we write
\begin{align*}
X=_{\textup{d}}\textup{Gumbel}(b,a).
\end{align*}
\end{definition*}

The following theorem describes the distribution of $T_{N}$ in terms of the Gumbel distribution. In this paper, $\to_{\textup{d}}$ denotes convergence in distribution \cite{billingsley2013}. That is, we write $X_{N}\to_{\dist}X$ as $N\to\infty$ if a sequence of random variables $\{X_{N}\}_{N\ge1}$ converges in distribution to $X$ as $N\to\infty$, which means
\begin{align}\label{cddef00}
\P(X_{N}\le x)
\to\P(X\le x)\quad\text{as }N\to\infty,
\end{align}
for all $x\in\R$ such that the function $F(x):=\P(X\le x)$ is continuous.

\begin{theorem}\label{gumbel}
Let $\{\tau_{n}\}_{n=1}^{\infty}$ be a sequence of iid realizations of any random variable $\tau$ satisfying
\begin{align*}
\P(\tau\le t)
\sim At^{p}e^{-C/t^{\beta}}\quad\text{as }t\to0+,
\end{align*}
for some constants $C>0$, $A>0$, $\beta\in(0,1]$, and $p\in\R$. If $T_{N}:=\min\{\tau_{1},\dots,\tau_{N}\}$, then
\begin{align}\label{cd}
\frac{T_{N}-b_{N}}{a_{N}}
\to_{\textup{d}}
X=_{\textup{d}}\textup{Gumbel}(0,1)\quad\text{as }N\to\infty,
\end{align}
where 
\begin{align}\label{abab}
\begin{split}
a_{N}
&
=\frac{b_{N}}{\beta\ln(AN)},\quad
b_{N}
=\Big(\frac{C}{\ln(AN)}\Big)^{1/\beta},\quad\text{if }p=0,\\
a_{N}
&
=\frac{b_{N}}{p(1+W)},\quad
b_{N}
=\Big(\frac{C\beta}{pW}\Big)^{1/\beta},\quad\text{if }p\neq0,
\end{split}
\end{align}
and
\begin{align}\label{dubdub}
W
&=
\begin{cases}
W_{0}\big((C\beta/p)(AN)^{\beta/p}\big) & \text{if }p>0,\\
W_{-1}\big((C\beta/p)(AN)^{\beta/p}\big) & \text{if }p<0,
\end{cases}
\end{align}
where $W_{0}(z)$ is the principal branch of the LambertW function and $W_{-1}(z)$ is the lower branch \cite{corless1996}. In addition, \eqref{cd} also holds for
\begin{align}\label{ababuf}
a_{N}
&=\frac{C^{1/\beta}}{\beta(\ln N)^{1+1/\beta}},\quad
b_{N}
=\Big(\frac{C}{\ln N}
+\frac{Cp\ln(\ln(N))}{\beta(\ln N)^{2}}
-\frac{C\ln(AC^{p/\beta})}{(\ln N)^{2}}\Big)^{1/\beta}.
\end{align}
\end{theorem}

The following theorem generalizes Theorem~\ref{gumbel} to the $k$th fastest FPT, $T_{k,N}$, defined in \eqref{tkn}.

\begin{theorem}\label{kth}
Under the assumptions of Theorem~\ref{gumbel} with $\{a_{N}\}_{N\ge1}$ and $\{b_{N}\}_{N\ge1}$ given by either \eqref{abab} or \eqref{ababuf}, the following convergence in distribution holds for any fixed $k\ge1$, 
\begin{align}\label{cd1}
\frac{T_{k,N}-b_{N}}{a_{N}}
\to_{\textup{d}}
X_{k}\quad\text{as }N\to\infty,
\end{align}
where the probability density function of $X_{k}$ is
\begin{align}\label{xkpdf}
\frac{\dd }{\dd x}\P(X_{k}\le x)
=\frac{\exp(kx-e^{x})}{(k-1)!},\quad\text{for all } x\in\R.
\end{align}
Furthermore, for any fixed $k\ge1$, the following convergence in distribution holds for the joint random variables,
\begin{align}\label{cd2}
\left(\frac{T_{1,N}-b_{N}}{a_{N}},\dots,\frac{T_{k,N}-b_{N}}{a_{N}}\right)
\to_{\textup{d}}
\mathbf{X}^{(k)}
=(X_{1},\dots,X_{k})\in\R^{k}\quad\text{as }N\to\infty,
\end{align}
where the joint probability density function of $\mathbf{X}^{(k)}\in\R^{k}$ is
\begin{align*}
f_{\mathbf{X}^{(k)}}(x_{1},\dots,x_{k})
&=\begin{cases}
\exp(-e^{x_{k}})\prod_{r=1}^{k}e^{x_{r}} & \text{if $x_{1}\le\dots\le x_{k}$},\\
0 & \text{otherwise}.
\end{cases}
\end{align*}
\end{theorem}


\subsection{Higher order moments}

Under the assumptions of Theorem~\ref{gumbel}, the next theorem gives higher order approximations to the moments of $T_{N}=T_{1,N}$ and $T_{k,N}$.

\begin{theorem}\label{moments}
Under the assumptions of Theorem~\ref{gumbel} with $\{a_{N}\}_{N\ge1}$ and $\{b_{N}\}_{N\ge1}$ given by either \eqref{abab} or \eqref{ababuf}, assume further that $\E[T_{N}]<\infty$ for some $N\ge1$.

Then for each moment $m\in(0,\infty)$, we have that
\begin{align*}
\E\left[\left(\frac{T_{N}-b_{N}}{a_{N}}\right)^{m}\right]
\to\E[X^{m}]\quad\text{as }N\to\infty,\quad\text{where }X=_{\textup{d}}\textup{Gumbel}(0,1),
\end{align*}
and thus
\begin{align*}
\E[(T_{N}-b_{N})^{m}]
=a_{N}^{m}\E[X^{m}]+o(a_{N}^{m}),
\end{align*}
where $f(N)=o(a_{N}^{m})$ means $\lim_{N\to\infty}a_{N}^{-m}f(N)=0$.

Furthermore, if $k\ge1$ and $m\in(0,\infty)$, then
\begin{align*}
\E\left[\left(\frac{T_{k,N}-b_{N}}{a_{N}}\right)^{m}\right]
\to\E[X_{k}^{m}]\quad\text{as }N\to\infty,
\end{align*}
where $X_{k}$ has the probability density function in \eqref{xkpdf}. Therefore,
\begin{align*}
\E[(T_{k,N}-b_{N})^{m}]
=a_{N}^{m}\E[X_{k}^{m}]+o(a_{N}^{m}).
\end{align*}
\end{theorem}

\begin{remark}\label{remarkmoments}
Theorem~\ref{moments} gives explicit higher order approximations to the integer moments of $T_{N}$ and $T_{k,N}$ since $\E[X_{k}^{m}]$ can be explicitly calculated if $m>0$ is an integer. For example, we have that as $N\to\infty$,
\begin{align*}
\E[T_{N}]
&=b_{N}-\gamma a_{N}+o(a_{N}),\\
\textup{Variance}(T_{N})
&=\frac{\pi^{2}}{6}a_{N}^{2}+o(a_{N}^{2}),\\
\E[T_{k,N}]
&=b_{N}+\psi(k)a_{N}+o(a_{N})
=\E[T_{1,N}]+H_{k-1}a_{N}+o(a_{N}),\\
\textup{Variance}(T_{k,N})
&=\psi'(k)a_{N}^{2}+o(a_{N}^{2}),
\end{align*}
where $\gamma\approx0.5772$ is the Euler-Mascheroni constant, $\psi(x)$ is the digamma function, and $H_{k-1}=\sum_{r=1}^{k-1}\tfrac{1}{r}$ is the $(k-1)$-th harmonic number. Furthermore, if we use the values of $a_{N}$ and $b_{N}$ in \eqref{ababuf}, then we obtain that as $N\to\infty$,
\begin{align*}
\E[T_{N}]
&=b_{N}-\gamma a_{N}+o(a_{N})\\
&=\Big(\frac{C}{\ln N}\Big)^{1/\beta}\left\{\Big(1
+\frac{p\ln(\ln(N))}{\beta\ln N}
-\frac{C\ln(AC^{p/\beta})}{\ln N}\Big)^{1/\beta}
-\frac{\gamma }{\beta\ln N}\right\}+o(a_{N}).
\end{align*}
Notice that the leading order term in this expression agrees with Theorem~\ref{uni}. The higher order terms show how the constants $A$ and $p$ affect the extreme statistics.

\end{remark}


\section{Examples}\label{examples}

In this section, we apply our results to several examples. The first three examples (sections~\ref{half}-\ref{pearson}) are in simple one-dimensional geometries. The third example (section~\ref{narrow}) is for the narrow escape problem with a small target on the boundary of a sphere. The final two examples (sections~\ref{sectiongen} and \ref{manifold}) show how our results hold in very general setups. We begin by introducing notation which is common to these examples. The notation and framework is similar to \cite{lawley2020uni} which studied FPTs of normal diffusion.

Let $\{X_{\alpha}(t)\}_{t\ge0}$ denote the position of a subdiffusive searcher with characteristic generalized diffusivity $K_{\alpha}>0$ on a $d$-dimensional manifold $M$. Let $p_{\alpha}(x,t\,|x_{0},0)$ denote the probability density that $X_{\alpha}(t)=x\in M$ given $X_{\alpha}(0)=x_{0}\in M$. As in section~\ref{prelim}, we construct the subdiffusive process by a random time change of a diffusive process. In particular, let $\{X_{1}(s)\}_{s\ge0}$ denote the position of a diffusive searcher on $M$ with probability density $p_{1}(x,s\,|\,x_{0},0)$. Further, let $\{U_{\alpha}(s)\}_{s\ge0}$ denote an $\alpha$-stable subordinator independent of $X_{1}$ with inverse $\{S_{\alpha}(t)\}_{t\ge0}$ (see \eqref{S}). We then define
\begin{align}\label{define}
X_{\alpha}(t):=X_{1}(S_{\alpha}(t)),\quad t\ge0.
\end{align}
We are most interested in the case $\alpha\in(0,1)$, though our results remain valid if we simply take $U_{\alpha}(s)=s$ and $S_{\alpha}(t)=t$ in the case $\alpha=1$.

Let $\tau>0$ be the subdiffusive FPT to some target ${\Omega_{\text{T}}}\subset M$,
\begin{align}\label{tau5}
\tau
:=\inf\{t>0:X_{\alpha}(t)\in \Omega_{\text{T}}\}
=U_{\alpha}(\sigma),
\end{align}
where $\sigma>0$ is the diffusive FPT,
\begin{align}\label{sigma5}
\sigma
:=\inf\{s>0:X_{1}(s)\in \Omega_{\text{T}}\}.
\end{align}
Let $\{\tau_{n}\}_{n=1}^{\infty}$ be a sequence of iid realizations of $\tau$ and let $T_{k,N}$ denote the $k$th order statistic defined in \eqref{tkn}. Assume the target ${\Omega_{\text{T}}}$ is the closure of its interior, which avoids trivial cases such as the target being a single point, which would force $\tau=\sigma=+\infty$ in dimension $d\ge2$. Assume the initial distribution of $X_{\alpha}(0)=X_{1}(0)$ is a probability measure with compact support $\Omega_{0}\subset M$ that does not intersect the target,
\begin{align}\label{away}
\Omega_{0}\cap {\Omega_{\text{T}}}=\varnothing.
\end{align}
As one example, the initial distribution could be a Dirac delta function at a single point $X_{\alpha}(0)=X_{1}(0)=x_{0}=\Omega_{0}\in M$ if $x_{0}\notin {\Omega_{\text{T}}}$. As another example, the initial distribution could be uniform on a closed set $\Omega_{0}$ satisfying \eqref{away}.

In each of the examples below, we find that for any $m\ge1$ and $k\ge1$, the $m$th moment of the $k$th fastest FPT satisfies
\begin{align}\label{leadinglater}
\E[(T_{k,N})^{m}]
\sim\bigg(\frac{t_{\alpha}}{(\ln N)^{2/\alpha-1}}\bigg)^{m}
\quad\text{as }N\to\infty,
\end{align}
where $t_{\alpha}>0$ is the characteristic timescale of subdiffusion,
\begin{align*}
t_{\alpha}
:=\Big(\alpha^{\alpha}(2-\alpha)^{2-\alpha}\frac{L^{2}}{4K_{\alpha}}\Big)^{1/\alpha}>0,
\end{align*}
where $K_{\alpha}>0$ is the generalized diffusion coefficient, and $L>0$ is a certain geodesic distance between $\Omega_{0}$ and the target $\Omega_{\text{T}}$. The distance $L>0$ is given below, but we note here that it is the shortest distance between the possible searcher starting locations $\Omega_{0}$ and the target $\Omega_{\text{T}}$ that avoids regions of slow diffusion and avoids reflecting obstacles. Further, the distance $L>0$ is unaffected by any drift on the searcher.

Furthermore, in some of these examples, we are able to find constants $A_{1}>0$, $p_{1}\in\R$, and $C_{1}=L^{2}/(4K_{\alpha})$ so that the diffusive FPT satisfies
\begin{align}\label{shorts}
\P(\sigma\le s)
\sim A_{1}s^{p_{1}}e^{-C_{1}/s}\quad\text{as }s\to0+.
\end{align}
We therefore apply Corollary~\ref{corlog} to conclude that the subdiffusive FPT satisfies
\begin{align}\label{shortt}
\P(\tau\le t)
\sim At^{p}e^{-C/t^{\beta}}\quad\text{as }t\to0+,
\end{align}
where
\begin{align}\label{values}
\beta
=\frac{\alpha}{2-\alpha},\quad
C
=(t_{\alpha})^{\beta},\quad
p=\beta p_{1},
\quad
A
=\frac{A_{1} \Big(\alpha ^{\frac{-\alpha }{2-\alpha }} C_{1}^{\frac{1-\alpha}{2-\alpha}}\Big)^{p_{1}}}{\sqrt{\alpha (2-\alpha)}}.
\end{align}
We then apply Theorem~\ref{gumbel} to conclude that
\begin{align}\label{cdlater}
\frac{T_{N}-b_{N}}{a_{N}}
\to_{\dist}\textup{Gumbel}(0,1)\quad\text{as }N\to\infty,
\end{align}
where $\{a_{N}\}_{N\ge1}$ and $\{b_{N}\}_{N\ge1}$ are given by \eqref{abab} or \eqref{ababuf}. The convergence in distribution in \eqref{cdlater} means that the distribution of $T_{N}$ is approximately Gumbel with shape parameter $b_{N}$ and scale parameter $a_{N}$,
\begin{align*}
\P(T_{N}>t)
\approx\exp\Big[-\exp\Big(\frac{t-b_{N}}{a_{N}}\Big)\Big]\quad\text{if }N\gg1.
\end{align*}
More generally, we apply Theorem~\ref{kth} to conclude that the $k$th fastest FPT has the following convergence in distribution,
\begin{align}\label{cdlaterk}
\frac{T_{k,N}-b_{N}}{a_{N}}
\to_{\dist}X_{k}\quad\text{as }N\to\infty,
\end{align}
where $X_{k}$ has the probability density function in \eqref{xkpdf}. Furthermore, we apply Theorem~\ref{moments} to obtain higher order corrections to the leading order moment behavior in \eqref{leadinglater}. For example,
\begin{align}\label{higherexample}
\begin{split}
\E[T_{N}]
&=b_{N}-\gamma a_{N}+o(a_{N})\quad\text{as }N\to\infty,\\
\E[T_{k,N}]
&=\E[T_{N}]+H_{k-1}a_{N}+o(a_{N})\quad\text{as }N\to\infty,
\end{split}
\end{align}
where $H_{k-1}=\sum_{r=1}^{k-1}\tfrac{1}{r}$ is the $(k-1)$-th harmonic number.

\subsection{Pure subdiffusion in one dimension}\label{half}

Consider a subdiffusive searcher $X_{\alpha}(t)$ on  $M=\R$ with an absorbing target $\Omega_{\text{T}}=(-\infty,0]$. The probability density $p_{\alpha}(x,t\,|x_{0},0)$ that $X_{\alpha}(t)=x>0$ given $X_{\alpha}(0)=x_{0}>0$ satisfies the fractional FPE,
\begin{align}\label{1d}
\begin{split}
\frac{\partial}{\partial t}p_{\alpha}
&= K_{\alpha}\frac{\partial^{2}}{\partial x^{2}}\D p_{\alpha},\quad x>0,\,t>0,\\
p_{\alpha}
&=\delta(x-x_{0}),\quad t=0,
\end{split}
\end{align}
with an absorbing boundary condition at the origin,
\begin{align}\label{abc}
p_{\alpha}
&=0,\quad x=0,\,t>0.
\end{align}

The corresponding diffusive process $X_{1}(s)$ used in \eqref{define} thus follows the SDE,
\begin{align*}
\dd X_{1}(s)
=\sqrt{2K_{\alpha}}\,\dd W(s),
\quad X_{1}(0)=x_{0}>0,
\end{align*}
where $\{W(s)\}_{s\ge0}$ is a standard Brownian motion. Assuming the initial condition is a Dirac mass at $X_{1}(0)=X_{\alpha}(0)=x_{0}>0$, it is well-known that the cumulative distribution function of the diffusive FPT $\sigma$ in \eqref{sigma5} is \cite{carslaw1959}
\begin{align}\label{erfc}
\P({\sigma}\le s)
=\text{erfc}\Big(\frac{x_{0}}{\sqrt{4K_{\alpha}s}}\Big),\quad s>0.
\end{align}
Taking $s\to0+$ in \eqref{erfc}, we find that $\P(\sigma\le s)$ satisfies \eqref{shorts} with
\begin{align*}
A_{1}
=\sqrt{\frac{4K_{\alpha}}{\pi x_{0}^{2}}},\quad
p_{1}
=\frac{1}{2},\quad
C_{1}
=\frac{x_{0}^{2}}{4K_{\alpha}}.
\end{align*}
Therefore, Corollary~\ref{corlog} ensures that the cumulative distribution function of the subdiffusive FPT ${\tau}$ in \eqref{tau5} satisfies \eqref{shortt} with $\beta,C,p,A$ given in \eqref{values}.

Using \eqref{erfc}, we have that $\P({\sigma}>s)=\O(s^{-1/2})$ as $s\to\infty$. Therefore, applying Theorems~\ref{uni} and \ref{nice} yields that the $m$th moment of $T_{k,N}$ has the leading order behavior in \eqref{leadinglater} where the geodesic distance is merely the distance to the target, $L=x_{0}$. Furthermore, applying Theorems~\ref{gumbel}-\ref{kth} yields the convergence in distribution of $T_{N}$ in \eqref{cdlater} and $T_{k,N}$ in \eqref{cdlaterk}, and we further have the higher order moment formulas of Theorem~\ref{moments}, which includes the mean formulas in \eqref{higherexample}.

\begin{figure}[t]
\centering
\includegraphics[width=1\linewidth]{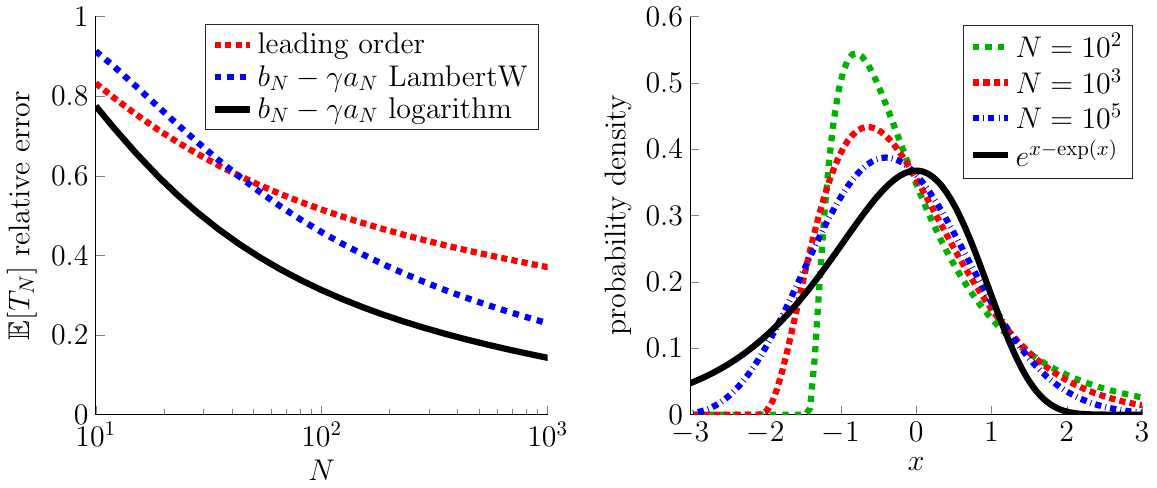}
\caption{One-dimensional subdiffusion with $\alpha=1/2$. The left panel plots the relative error \eqref{re} for various approximations to the mean fastest FPT, $\E[T_{N}]$. The right panel illustrates the convergence in distribution of $(T_{N}-b_{N})/a_{N}$ to a standard Gumbel random variable as $N$ grows. See the text for more details.}
\label{figboth}
\end{figure}

We illustrate some of these results numerically in Figure~\ref{figboth} in the case $\alpha=1/2$. In the left panel, we plot the relative error
\begin{align}\label{re}
\Big|\frac{\E[T_{N}]-\T_{N}}{\E[T_{N}]}\Big|,
\end{align}
where $\T_{N}$ is one of three approximations of $\E[T_{N}]$ described below. The value of $\E[T_{N}]$ used in \eqref{re} is calculated numerically by quadrature,
\begin{align*}
\E[T_{N}]
=\int_{0}^{\infty}(\P(\tau>t))^{N}\,\dd t,
\end{align*}
where $\P(\tau>t)$ is given by the analytical formula
\begin{align}\label{bigformula}
\begin{split}
&\P(\tau>t)
=\int_{0}^{\infty}q_{\alpha}(s,t)\P(\sigma>s)\,\dd s\\
&\quad=\frac{\Gamma \left(-\frac{1}{4}\right) \, _1F_3\left(\frac{3}{4};\frac{5}{4},\frac{3}{2},\frac{7}{4};-\frac{1}{256 t}\right)-24 \sqrt{t} \Gamma \left(\frac{1}{4}\right) \, _1F_3\left(\frac{1}{4};\frac{1}{2},\frac{3}{4},\frac{5}{4};-\frac{1}{256 t}\right)}{24 \sqrt{2} \pi  t^{3/4}}\\
&\qquad+\frac{\, _1F_3\left(\frac{1}{2};\frac{3}{4},\frac{5}{4},\frac{3}{2};-\frac{1}{256 t}\right)}{2 \sqrt{\pi } \sqrt{t}}+1,
\end{split}
\end{align}
involving the gamma function, $\Gamma(\cdot)$, and the hypergeometric function,
\begin{align*}
{}_pF_q(a_1,\ldots,a_p;b_1,\ldots,b_q;z) = \sum_{n=0}^\infty \frac{(a_1)_n\cdots(a_p)_n}{(b_1)_n\cdots(b_q)_n} \, \frac {z^n} {n!},
\end{align*}
and the rising factorial (or Pochhammer symbol),
\begin{align*}
(a)_0 &= 1, \quad (a)_n = a(a+1)(a+2) \cdots (a+n-1), \quad n \geq 1.
\end{align*}
The integration in \eqref{bigformula} was done using Mathematica \cite{mathematica} and the fact that $q_{\alpha}(s,t)=\frac{t}{\alpha s^{1+1/\alpha}}l_{\alpha}(ts^{-1/\alpha})$ where $l_{\alpha}(z)=\frac{e^{-\frac{1}{4 z}}}{2 \sqrt{\pi } z^{3/2}}$ since $\alpha=1/2$.

The red dotted curve in the left panel of Figure~\ref{figboth} is the error \eqref{re} for the leading order approximation $\T_{N}=t_{\alpha}/(\ln N)^{2/\alpha-1}$ in \eqref{leadinglater}. The blue dashed curve (respectively black solid curve) is for the higher order approximation $\T_{N}=b_{N}-\gamma a_{N}$ where $a_{N}$ and $b_{N}$ are given by \eqref{abab} (respectively \eqref{ababuf}). In agreement with the theory, the error decays faster for the higher order approximations than for the leading order approximation.

In the right panel of Figure~\ref{figboth}, we show the convergence in distribution of $(T_{N}-b_{N})/a_{N}$ to a standard Gumbel random variable where $a_{N}$ and $b_{N}$ are in \eqref{abab}.  The colored, non-solid curves are the probability density function of $(T_{N}-b_{N})/a_{N}$ for $N\in\{10^{2},10^{3},10^{5}\}$ which was calculated using the analytical formula for the survival probability of $\tau$ in \eqref{bigformula}. As $N$ increases, these curves approach the probability density function of a standard Gumbel random variable ($e^{x-\exp(x)}$). We set $x_{0}=L=K_{\alpha}=1$ in Figure~\ref{figboth}.

\subsection{Partially absorbing target}

We can quickly extend the analysis above to the case that the target is partially absorbing \cite{grebenkov2006}. In this case, the probability density $p_{\alpha}$ for the subdiffusive process again satisfies \eqref{1d}, with the absorbing boundary condition at the origin in \eqref{abc} replaced by the partially absorbing condition \cite{seki2003, eaves2008}
\begin{align}\label{partial}
K_{\alpha}\frac{\partial}{\partial x}p_{\alpha}
&=\kappa_{\alpha}p_{\alpha},\quad x=0,\,t>0,
\end{align}
where $\kappa_{\alpha}>0$ is a generalized target reactivity parameter with dimensions $(\text{length})(\text{time})^{-\alpha}$.

To construct the subdiffusive process, let $\{X_{1}(s)\}_{s\ge0}$ be a diffusion with probability density $p_{1}(x,s\,|\,x_{0},0)$ satisfying the integer FPE,
\begin{align*}
\frac{\partial}{\partial s}p_{1}
&=K_{\alpha}\frac{\partial^{2}}{\partial x^{2}}p_{1},\quad x>0,\,s>0,\\
K_{\alpha}\frac{\partial}{\partial x}p_{1}
&=\kappa_{\alpha}p_{1},\quad x=0,\,s>0,\\
p_{1}
&=\delta(x-x_{0}),\quad s=0.
\end{align*}
Defining $X_{\alpha}(t):=X_{1}(S_{\alpha}(t))$, it is immediate that the density of $X_{\alpha}$ satisfies \eqref{1d} with the partially absorbing boundary condition \eqref{partial}. Furthermore, if ${\sigma}$ is the absorption time of $X_{1}(s)$ at the origin, then the absorption time of $X_{\alpha}(t)$ at the origin is ${\tau}=U_{\alpha}({\sigma})$. Assuming $X_{1}(0)=X_{\alpha}(0)=x_{0}>0$, it is well-known that \cite{carslaw1959}
\begin{align}\label{erfck}
\P({\sigma}\le s)
=\text{erfc}\Big(\frac{{{x_{0}}}}{\sqrt{4K_{\alpha} s}}\Big)-e^{\frac{\kappa_{\alpha}  (\kappa_{\alpha}  s+{{x_{0}}})}{K_{\alpha}}} \text{erfc}\Big(\frac{2 \kappa_{\alpha}  s+{{x_{0}}}}{\sqrt{4K_{\alpha} s}}\Big),\quad s>0,
\end{align}
and therefore $\P(\sigma\le s)$ satisfies \eqref{shorts} with
\begin{align}\label{kvalues}
A_{1}
=\frac{4}{\sqrt{\pi}}\frac{\kappa_{\alpha}x_{0}}{K_{\alpha}}\Big(\frac{K_{\alpha}}{x_{0}^{2}}\Big)^{3/2},\quad
p_{1}
=\frac{3}{2},\quad
C_{1}
=\frac{x_{0}^{2}}{4K_{\alpha}}.
\end{align}
By Corollary~\ref{corlog}, the cumulative distribution function of ${\tau}$ satisfies \eqref{shortt} with $\beta,C,p,A$ given in \eqref{values}. A straightforward calculation shows that \eqref{shortt}-\eqref{values} and \eqref{kvalues} agree with the results of Grebenkov in \cite{grebenkov2010} in this example.

Using \eqref{erfck}, we have that $\P({\sigma}>s)=\O(s^{-1/2})$ as $s\to\infty$. Therefore, Theorems~\ref{uni} and \ref{nice} imply that the $m$th moment of $T_{k,N}$ has the leading order behavior in \eqref{leadinglater} with $L=x_{0}$. We emphasize that this leading order behavior is independent of the partial absorption rate $\kappa_{\alpha}>0$ and is the same as that found in the section above for a perfectly absorbing target ($\kappa_{\alpha}=\infty$).

To see the affect of $\kappa_{\alpha}$ at higher order, we apply the analysis of section~\ref{distribution}. In particular, applying Theorems~\ref{gumbel}-\ref{kth} yields the convergence in distribution of $T_{N}$ in \eqref{cdlater} and $T_{k,N}$ in \eqref{cdlaterk}. Further, the higher order moment formulas of Theorem~\ref{moments} give as $N\to\infty$ (see Remark~\ref{remarkmoments})
\begin{align*}
\E[T_{N}]
&=\Big(\frac{C}{\ln N}\Big)^{1/\beta}\left\{\Big(1
+\frac{p\ln(\ln(N))}{\beta\ln N}
-\frac{C\ln(AC^{p/\beta})}{\ln N}\Big)^{1/\beta}
-\frac{\gamma }{\beta\ln N}\right\}+o(a_{N}),
\end{align*}
where we have used the values \eqref{ababuf} for $a_{N}$ and $b_{N}$. Noting that $A$ is linear in $\kappa_{\alpha}$ (see \eqref{values} and \eqref{kvalues}), this expression shows how the finite reactivity $\kappa_{\alpha}$ affects $\E[T_{N}]$ at higher order.

\subsection{Subdiffusion in one dimension with a drift}\label{pearson}

Consider a subdiffusive searcher $X_{\alpha}(t)$ on $M=\R$ with absorbing boundary conditions at the targets at $x=0$ and $x={L_{0}}$ (meaning the target is $\Omega_{\text{T}}=(-\infty,0]\cup[{L_{0}},\infty)$). Suppose further that there is a constant drift that pushes the searcher toward $x={L_{0}}$. The probability density $p_{\alpha}(x,t\,|x_{0},0)$ that $X_{\alpha}(t)=x$ given $X_{\alpha}(0)=x_{0}\in(0,{L_{0}})$ satisfies the fractional FPE,
\begin{align*}
\frac{\partial}{\partial t}p_{\alpha}
&=\Big(-\frac{\partial}{\partial x}\big[V_{\alpha}p\big]+K_{\alpha}\frac{\partial^{2}}{\partial x^{2}}\Big)\D p_{\alpha},\quad x>0,\,t>0,\\
p_{\alpha}
&=\delta(x-x_{0}),\quad t=0,\\
p_{\alpha}
&=0,\quad x\in\{0,{L_{0}}\},\,t>0,
\end{align*}
where $V_{\alpha}>0$ is a constant with dimension $(\text{length})(\text{time})^{-\alpha}$.

To construct this subdiffusive process, the diffusion $X_{1}(s)$ satisfies the SDE,
\begin{align*}
\dd X_{1}(s)
=V_{\alpha}\,\dd s+\sqrt{2K_{\alpha}}\,\dd W(s),\quad X_{1}(0)=x_{0}\in(0,{L_{0}}).
\end{align*}
If $X_{1}(0)=X_{\alpha}(0)=x_{0}\in(0,{L_{0}}/2)$ (so that the searcher starts closer to the target at $x=0$), then the cumulative distribution function of the diffusive FPT ${\sigma}$ in \eqref{sigma5} has the short-time behavior in \eqref{shorts} where (see the Appendix),
\begin{align*}
A_{1}
=\sqrt{\frac{4K_{\alpha}}{\pi x_{0}^{2}}}
\exp\Big(-\frac{V_{\alpha}}{2K_{\alpha}}x_{0}\Big),\quad
p_{1}
=\frac{1}{2},\quad
C_{1}
=\frac{x_{0}^{2}}{4K_{\alpha}}.
\end{align*}
Therefore, Corollary~\ref{corlog} ensures that the cumulative distribution function of the subdiffusive FPT ${\tau}$ in \eqref{tau5} satisfies \eqref{shortt} with $\beta,C,p,A$ given in \eqref{values}.

It is well-known that $\P(\sigma>s)$ vanishes exponentially as $s\to\infty$, and therefore Theorem~\ref{nice} ensures that $\E[T_{N}]<\infty$ for $N>1/\alpha$. Hence, we can again apply Theorems~\ref{uni}, \ref{gumbel}, \ref{kth}, and \ref{moments} to obtain the large $N$ distribution and moments of $T_{k,N}$ (and so \eqref{leadinglater} holds with $L=x_{0}\in(0,{L_{0}}/2)$). We note that these results show that the drift $V_{\alpha}$ has no effect on the leading order extreme statistics as $N\to\infty$, and Theorem~\ref{moments} shows how the drift affects extreme statistics at higher order.

\subsection{Narrow escape for subdiffusion}\label{narrow}

The narrow escape problem is to determine the time it takes a single diffusive searcher to find a small target in an otherwise reflecting bounded domain \cite{holcman2014, ward10, ward10b, grebenkov2017}. The vast majority of works on the narrow escape problem focus on the mean of this time. Recently, Grebenkov, Metzler, and Oshanin found an approximation for the full distribution of this FPT in a spherical domain \cite{grebenkov2019}. In this section, we use their results to determine the full distribution and moments for the fastest subdiffusive FPT in the narrow escape problem.

Consider a subdiffusive searcher $X_{\alpha}(t)$ in the three-dimensional sphere of radius $L>0$,
\begin{align*}
M
=\{x\in\R^{3}:\|x\|<L\},
\end{align*}
with a reflecting boundary. Suppose the target $\partial\Omega_{\text{T}}$ is a small spherical cap with polar angle $\eps>0$. Assuming $X_{1}(0)=0$, it was recently shown that the probability density of the diffusive FPT ${\sigma}$ in \eqref{sigma5} has the short-time behavior (see equation (C.16) in \cite{grebenkov2019})
\begin{align*}
\frac{\dd}{\dd s}\P(\sigma\le s)
\sim\frac{L(1-\cos\eps)}{\sqrt{4\pi K_{\alpha}s^{3}}}\Big(\frac{L^{2}}{2K_{\alpha}s}\Big)e^{-L^{2}/(4K_{\alpha}s)}\quad\text{as }s\to0+.
\end{align*}
Taking the Laplace transform of this expression, dividing by the Laplace variable, and taking the inverse Laplace transform yields that $\P(\sigma\le s)$ has the short-time behavior in \eqref{shorts} with
\begin{align*}
A_{1}
=\frac{L (1-\cos \eps )}{\sqrt{\pi K_{\alpha}}},\quad
p_{1}
=-1/2,\quad
C_{1}
=\frac{L^{2}}{4K_{\alpha}}.
\end{align*}
We therefore apply Corollary~\ref{corlog} to obtain the short-time behavior of the subdiffusive FPT $\tau$ in \eqref{tau5} for the narrow escape problem, with $\beta,A,p,C$ given in \eqref{values}.

It is well-known that $\P(\sigma>s)$ vanishes exponentially as $s\to\infty$, and therefore Theorem~\ref{nice} ensures that $\E[T_{N}]<\infty$ for $N>1/\alpha$. Hence, we can again apply Theorems~\ref{uni}, \ref{gumbel}, \ref{kth}, and \ref{moments} to obtain the large $N$ distribution and moments of $T_{k,N}$. We note that a single diffusive FPT $\sigma$ and a single subdiffusive FPT $\tau$ both diverge as the hole size vanishes (i.e.\ the narrow escape limit, $\eps\to0$). However, these results show that the size of the hole, $\eps>0$, has no effect on the leading order extreme statistics as $N\to\infty$. In particular, the leading order extreme statistics for this narrow escape problem are identical to the case that the entire boundary is an absorbing target. Theorem~\ref{moments} shows how the target size $\eps$ affects extreme statistics at higher order.

\subsection{Subdiffusion in $\R^{d}$ with space-dependent drift and diffusivity}\label{sectiongen}

Let $X_{\alpha}(t)$ be a $d$-dimensional subdiffusive process with a general space-dependent drift and diffusivity. Specifically, suppose the probability density of $X_{\alpha}$ satisfies the fractional FPE in \eqref{ffpe}. We construct this process by setting $X_{\alpha}(t):=X_{1}(S_{\alpha}(t))$ exactly as in section~\ref{prelim}.

Suppose the target $\Omega_{\text{T}}\subset\R^{d}$ is such that the complement of the target, $\R^{d}\backslash\Omega_{\text{T}}$, is bounded. It was proven in \cite{lawley2020uni} that the distribution of the diffusive FPT $\sigma$ in \eqref{sigma5} satisfies
\begin{align}\label{loggen}
\lim_{s\to0+}s\ln\P({\sigma}\le s)
=-\frac{L^{2}}{4K_{\alpha}}>0,
\end{align}
where
\begin{align}\label{LLL}
L
:=\inf_{x_{0}\in\Omega_{0},x\in\Omega_{T}}\drie(x_{0},x),
\end{align}
where $\drie$ is the geodesic distance defined in \eqref{drie}. Upon noting that $\P(\sigma> s)$ decays exponentially as $s\to\infty$ since $\R^{d}\backslash\Omega_{\text{T}}$ is bounded, we apply Corollary~\ref{corlog} and Theorems~\ref{uni} and \ref{nice} to find that the extreme statistics $T_{k,N}$ satisfy \eqref{leadinglater}. Note that the Riemannian metric in \eqref{ll} that defines the geodesic distance in \eqref{drie} and \eqref{LLL} does not depend on the drift in the fractional FPE \eqref{ffpe}. Hence, the extreme statistics of subdiffusion are independent of the drift to leading order as $N\to\infty$. Looking again at the Riemannian metric in \eqref{ll}, we see that a space-dependent diffusivity affects extreme statistics of subdiffusion by making the fastest searchers avoid regions of slow diffusivity.

\subsection{Subdiffusion on a manifold with reflecting obstacles}\label{manifold}

Finally, we consider the case of subdiffusion on a $d$-dimensional Riemannian manifold $M$ that is smooth, connected, and compact. We define $X_{\alpha}$ as in \eqref{define}, where $X_{1}(s)$ is a diffusion on $M$ described by its generator $\L^{*}$, which in each coordinate chart is a second order differential operator of the form
\begin{align*}
\L^{*} f
=K_{\alpha}\sum_{i,j=1}^{d}\frac{\partial}{\partial x_{i}}\Big(a_{ij}(x)\frac{\partial f}{\partial x_{j}}\Big),
\end{align*}
where the matrix $a=\{a_{ij}\}_{i,j=1}^{d}$ satisfies mild conditions (in each coordinate chart, $a$ is symmetric, continuous, and its eigenvalues are bounded above some $\gamma_{1}>0$ and bounded below some $\gamma_{2}>\gamma_{1}$). If $M$ has a boundary, then we assume $X_{1}$ (and therefore $X_{\alpha}$) reflects from the boundary.

One motivating example for this setup is the case that $M$ is a set in $\R^{{d}}$ with smooth outer and inner boundaries (the boundaries act as obstacles to the motion of the searcher). Alternatively, $M$ could be the 2-dimensional surface of a 3-dimensional sphere. Note that the narrow escape problem of a small target(s) in an otherwise reflecting bounded domain fits this setup.

In this setup, it was proven in \cite{lawley2020uni} that the distribution of the diffusive FPT $\sigma$ in \eqref{sigma5} satisfies \eqref{loggen} where $L$ is given by \eqref{LLL} and $\drie$ is the geodesic distance defined in \eqref{drie}, where the infimum is over smooth paths $\omega:[0,1]\to M$ which connect $\omega(0)=x_{0}$ to $\omega(1)=x$. Upon noting that $\P(\sigma> s)$ decays exponentially as $s\to\infty$ since $M$ is connected and compact, we apply Corollary~\ref{corlog} and Theorems~\ref{uni} and \ref{nice} to find that the extreme statistics $T_{k,N}$ satisfy \eqref{leadinglater}. Since the infimum in the definition of $\drie$ is over paths which lie in $M$, this shows that the fastest subdiffusive searchers take the shortest path to the target while avoiding any reflecting obstacles.

\section{Discussion}

In this paper, we investigated extreme statistics of anomalous subdiffusion. We found an explicit formula for the moments of the $k$th fastest FPT, $T_{k,N}$, out of $N\gg1$ subdiffusive searchers that holds in significant generality. While the mean FPT of a single subdiffusive searcher is typically infinite \cite{yuste2004}, we found that the fastest subdiffusive FPT has a finite mean if $N$ is sufficiently large. We further found an approximation of the distribution of $T_{k,N}$ and higher order moment approximations. A key step in our analysis was proving a relation between short-time distributions of diffusion and subdiffusion, which is akin to Varadhan's formula \cite{varadhan1967} in large deviation theory. We proved this relation for probability densities of the position of a random searcher (see Corollaries~\ref{corv} and \ref{corg}) and for the distribution of FPTs of random searchers (see Corollary~\ref{corlog}). This relation allowed us to employ methods recently developed to study extreme FPTs of diffusive searchers \cite{lawley2020uni, lawley2020dist}.

Our analysis yielded the counterintuitive result that extreme FPTs are faster for subdiffusive searchers compared to diffusive searchers. This result bears some resemblance to the interesting work of Guigas and Weiss \cite{guigas2008}, which used computer simulations to show that a subdiffusive searcher can quickly find a nearby target with a much higher probability than a diffusive searcher. Based on this computational result, it was claimed in \cite{guigas2008} that this identifies a way in which cells can benefit from their crowded internal state and the induced subdiffusion. { While Ref.~\cite{guigas2008} modeled subdiffusion by fractional Brownian motion, our mathematical analysis shows that the basic computational result of \cite{guigas2008} also holds for subdiffusion modeled by fractional FPEs.}  
{ Let $\sigma_{\text{norm}}$  and $\tau_{\text{sub}}$ denote the respective FPTs of a normal diffusive searcher and a subdiffusive searcher (modeled by a fractional FPE) to a target.} We found that the subdiffusive searcher has a much higher probability of finding the target before a small time $t$,
\begin{align}\label{simple}
\P(\sigma_{\text{norm}}\le t)
\approx
e^{-t_{1}/t}
\ll
\P(\tau_{\text{sub}}\le t)
\approx
e^{-(t_{\alpha}/t)^{\alpha/(2-\alpha)}},
\end{align}
where $t_{1}>0$ and $t_{\alpha}>0$ are the diffusive and subdiffusive timescales,
\begin{align*}
t_{1}
:=\frac{L^{2}}{4K_{1}},\quad
t_{\alpha}
:=\Big(\alpha^{\alpha}(2-\alpha)^{2-\alpha}\frac{L^{2}}{4K_{\alpha}}\Big)^{1/\alpha},\quad \alpha\in(0,1),
\end{align*}
where $L>0$ is a certain geodesic distance between the searcher starting locations and the target and $K_{1}$ and $K_{\alpha}$ are the diffusivity and generalized diffusivity (see Corollary~\ref{corlog} and section~\ref{sectiongen} for a precise meaning of \eqref{simple}). In fact, \eqref{simple} has been shown in certain exactly solvable geometries \cite{grebenkov2010} and can be anticipated from the well-known behavior of the propagator for pure subdiffusion in free space \cite{metzler2000}.

Another very interesting related work is Reference \cite{grebenkov2010jcp}, in which Grebenkov studied (sub)diffusing particles searching for a partially absorbing target, { where the subdiffusion is modeled by the fractional diffusion equation}. In contrast to the present work, Grebenkov assumed that the $N\gg1$ subdiffusive searchers are initially uniformly distributed in a volume $V$, and considered the thermodynamic limit of $V\to\infty$ and $N\to\infty$ with a fixed density $N/V$. The author then studied the short-time and long-time asymptotic behavior of the survival probability of the fastest searcher.

An important assumption in the present work is that the searchers cannot start arbitrarily close to the target, which precludes the case that the searchers are initially uniformly distributed in the entire domain (see \eqref{away} for a precise statement). If this assumption is removed, then the short-time distribution of the FPT of a single searcher and the resulting extreme FPT statistics are fundamentally different. Indeed, if the (sub)diffusive searchers are initially distributed uniformly in a $d$-dimensional sphere of radius $L>0$, then the FPT distribution of a single searcher to reach the boundary of the sphere has the short-time behavior \cite{grebenkov2010},
\begin{align*}
\P(\tau\le t\,|\,X_{\alpha}(0)=_{\dist}\textup{uniform})
\sim At^{\alpha/2}\quad\text{as }t\to0+,
\end{align*}
where $A=\frac{2}{\alpha\Gamma(\alpha/2)}\frac{d}{L}\sqrt{K_{\alpha}}$. It then follows from Theorems~2 and 3 in \cite{lawley2020comp} that the fastest FPT out of these $N\gg1$ uniformly distributed searchers is approximately Weibull with scale parameter $(AN)^{-2/\alpha}$ and shape parameter $\alpha/2>0$. In particular, the mean fastest FPT satisfies
\begin{align*}
\E[T_{N}\,|\,\text{uniformly distributed searchers}]
\sim\frac{\Gamma(1+2/\alpha)}{A^{2/\alpha}}\frac{1}{N^{2/\alpha}}\quad\text{as }N\to\infty.
\end{align*}
Interestingly, this again shows that the fastest subdiffusive searchers ($\alpha\in(0,1)$) find the target faster than the fastest diffusive searchers ($\alpha=1$).

\subsubsection*{Acknowledgments}
The author was supported by the National Science Foundation (Grant Nos.\ DMS-1944574, DMS-1814832, and DMS-1148230).

\section{Appendix}

\subsection{Proofs}


In this section of the appendix, we collect the proofs of the results in sections~\ref{varadhan} and \ref{extreme}. We begin with a lemma.

\begin{lemma}\label{exact}
Assume that
\begin{align*}
F_{1}(s)
&= A_{1}s^{p_{1}}e^{-C_{1}/s},\\
l(z)
&= Bz^{-{\xi}}e^{-\kappa z^{-{\theta}}},\\
q(s,t)
&=\frac{t}{\alpha s^{1+1/\alpha}}l(ts^{-1/\alpha}),
\end{align*}
where $\alpha>0$, $A_{1}>0$, $p_{1}\in\R$, $C_{1}>0$, $B>0$, ${\xi}\in\R$, $\kappa>0$, ${\theta}>0$.
Then for any $r_{1}>\alpha\theta/(\alpha+\theta)$ and $\eps_{2}>0$, we have that
\begin{align*}
F_{\alpha}(t)
:=\int_{0}^{\infty}q(s,t)F_{1}(s)\,\dd s
&\sim\int_{t^{r_{1}}}^{\eps_{2}}q(s,t)F_{1}(s)\,\dd s
\sim At^{p}e^{-C/t^{\beta}}\quad\text{as }t\to0+,
\end{align*}
where $\beta$, $C$, $A$, and $p$ are in \eqref{bcdef} and \eqref{complicated2}.
\end{lemma}

\begin{proof}[Proof of Lemma~\ref{exact}]

By assumption, we have that
\begin{align}
&\int_{0}^{\infty}q(s,t)F_{1}(s)\,\dd s\nonumber\\
&=\int_{0}^{\infty}
\frac{t}{\alpha s^{1+1/\alpha}}B(ts^{-1/\alpha})^{-{\xi}}\exp\Big[-\kappa t^{-{\theta}}s^{{\theta}/\alpha}\Big]A_{1}s^{p_{1}}
e^{-C_{1}/s}\,\dd s\nonumber\\
&=A_{1}B\alpha^{-1}t^{1-{\xi}}\int_{0}^{\infty}
s^{r}\exp\Big[-\kappa t^{-{\theta}}s^{{\theta}/\alpha}-C_{1}s^{-1}\Big]\,\dd s,\label{integral}
\end{align}
where $r:=p_{1}+{\xi}/\alpha-1-1/\alpha$. A simple calculus exercise shows that the maximum of the exponential factor in the integrand occurs at
\begin{align}\label{maxst}
s=s_{0}(t)
:=\Big(\frac{C_{1}\alpha}{\kappa{\theta}}\Big)^{\frac{\alpha}{\alpha+{\theta}}}
t^{\frac{\alpha{\theta}}{\alpha+{\theta}}}
:=s_{1}t^{\frac{\alpha{\theta}}{\alpha+{\theta}}}
>0,
\end{align}
where $s_{1}:=(\frac{C_{1}\alpha}{\kappa{\theta}})^{\frac{\alpha}{\alpha+{\theta}}}>0$. We thus introduce the change of variables,
\begin{align*}
u
=\frac{s}{s_{0}(t)},
\end{align*}
so that \eqref{integral} becomes
\begin{align*}
&A_{1}B\alpha^{-1}t^{1-{\xi}}(s_{0}(t))^{r+1}
\int_{0}^{\infty}
u^{r}
\exp\Big[-\kappa t^{-{\theta}}(s_{0}(t))^{{\theta}/\alpha}u^{{\theta}/\alpha}-C_{1}(s_{0}(t))^{-1}u^{-1}\Big]\,\dd u\\
&=A_{1}B\alpha^{-1}t^{1-{\xi}}(s_{0}(t))^{r+1}
\int_{0}^{\infty}
u^{r}
\exp\Big[-t^{\frac{\alpha{\theta}}{\alpha+{\theta}}}\Big(\kappa s_{1}^{{\theta}/\alpha}u^{{\theta}/\alpha}+C_{1}s_{1}^{-1}u^{-1}\Big)\Big]\,\dd u,
\end{align*}
and the maximum of the exponential occurs at $u=1$.

We thus apply Laplace's method and obtain
\begin{align*}
\int_{0}^{\infty}
f(u)
e^{x\phi(u)}\,\dd u
\sim
\frac{\sqrt{2\pi}f(1)e^{x\phi(1)}}{\sqrt{-x\phi''(1)}},\quad\text{as }x:=t^{\frac{-\alpha{\theta}}{\alpha+{\theta}}}\to\infty,
\end{align*}
where
\begin{align*}
f(u)
&:=u^{r},\quad
\phi(u)
:=-\kappa s_{1}^{{\theta}/\alpha}u^{{\theta}/\alpha}-C_{1}s_{1}^{-1}u^{-1},
\end{align*}
and thus
\begin{align*}
f(1)=1,\quad
\phi(1)=-C_{1}\big(\frac{\kappa{\theta}}{C_{1}\alpha}\big)^{\frac{\alpha}{\alpha+{\theta}}}\frac{(\alpha+{\theta})}{{\theta}},\quad
\phi''(1)=-C_{1}\big(\frac{\kappa{\theta}}{C_{1}\alpha}\big)^{\frac{\alpha}{\alpha+{\theta}}}\frac{(\alpha+{\theta})}{\alpha}.
\end{align*}
Putting this all together, we obtain
\begin{align*}
\int_{0}^{\infty}q_{\alpha}(s,t)F_{1}(s)\,\dd s
\sim
At^{p}e^{-Ct^{-\beta}}\quad\text{as }t\to0+,
\end{align*}
where $A,p,C,\beta$ are in \eqref{bcdef} and \eqref{complicated2}. To complete the proof, we need only check that
\begin{align*}
\int_{0}^{\infty}q_{\alpha}(s,t)F_{1}(s)\,\dd s
\sim\int_{t^{r_{1}}}^{\eps_{2}}q_{\alpha}(s,t)F_{1}(s)\,\dd s\quad\text{as }t\to0+,
\end{align*}
for any $r>\alpha\theta/(\alpha+\theta)$ and $\eps_{2}>0$. This follows from \eqref{maxst}, since the value of $s$ in \eqref{maxst} lies in the interval $(t^{r_{1}},\eps_{2})$ for small $t$ since $r_{1}>\alpha\theta/(\alpha+\theta)$.
\end{proof}

\begin{proof}[Proof of Theorem~\ref{log}]
First, define $g_{F}(s)=\ln(e^{C_{1}/t}F_{1}(s))$ and $g_{l}(z)=\ln(e^{\kappa z^{-\theta}}l(z))$ so that
\begin{align*}
F_{1}(s)
=e^{-C_{1}/s}e^{g_{F}(s)},\quad
l(z)
=e^{-\kappa z^{-{\theta}}}e^{g_{l}(z)}.
\end{align*}
Therefore, \eqref{logF} and \eqref{logl} ensure that
\begin{align}\label{crush}
\lim_{s\to0+}sg_{F}(s)
=\lim_{z\to0+}z^{{\theta}}g_{l}(z)
=0.
\end{align}
Next, define
\begin{align*}
q(s,t)
:=\frac{t}{\alpha s^{1+1/\alpha}}l(t/s^{1/\alpha}),
\end{align*}
and decompose $F_{\alpha}(t)$ into three integrals,
\begin{align}
F_{\alpha}(t)
&=\int_{0}^{t^{r_{1}}}q(s,t)F_{1}(s)\,\dd s
+\int_{t^{r_{1}}}^{\eps_{2}}q(s,t)F_{1}(s)\,\dd s
+\int_{\eps_{2}}^{\infty}q(s,t)F_{1}(s)\,\dd s\nonumber\\
&=:I_{1}+I_{2}+I_{3},\label{integrals0}
\end{align}
where $\eps_{2}>0$ and the exponent $r_{1}$ is such that
\begin{align}\label{r10}
0<
\frac{\alpha{\theta}}{\alpha+{\theta}}
<r_{1}
<\min\{\alpha,\theta\}.
\end{align}
We can choose $r_{1}$ satisfying \eqref{r10} since ${\theta}>0$ and $\alpha>0$.

Looking to the first integral in \eqref{integrals0}, it follows from \eqref{crush} that for sufficiently small $t>0$,
\begin{align*}
I_{1}
=\int_{0}^{t^{r_{1}}}q(s,t)e^{-C_{1}/s}e^{g_{F}(s)}\,\dd s
&\le e^{-(C_{1}/2)/t^{r_{1}}}\int_{0}^{t^{r_{1}}}q(s,t)\,\dd s.
\end{align*}
Now, changing variables $z=t/s^{1/\alpha}$ gives
\begin{align}\label{qint}
\int_{0}^{t^{r_{1}}}q(s,t)\,\dd s
\le\int_{0}^{\infty}\frac{t}{\alpha s^{1+1/\alpha}}l(t/s^{1/\alpha})\,\dd s
=\int_{0}^{\infty}l(z)\,\dd z:=B_{1}
<\infty,
\end{align}
since $l(z)$ is integrable by assumption. Therefore, we have that
\begin{align}\label{b1}
I_{1}
\le e^{-(C_{1}/2)/t^{r_{1}}}B_{1},\quad\text{for sufficiently small }t.
\end{align}

Moving to the third integral in \eqref{integrals0}, it follows from \eqref{crush} that for sufficiently small $t>0$,
\begin{align*}
I_{3}
& =\frac{t}{\alpha}\int_{\eps_{2}}^{\infty}\frac{1}{s^{1+1/\alpha}}\exp\Big[-\kappa\Big(\frac{s^{1/\alpha}}{t}\Big)^{{\theta}}\Big]e^{g_{l}(t/s^{1/\alpha})}F_{1}(s)\,\dd s\\
& \le\frac{t}{\alpha}\exp\Big[-\frac{\kappa}{2}\Big(\frac{\eps_{2}^{1/\alpha}}{t}\Big)^{{\theta}}\Big]\int_{\eps_{2}}^{\infty}\frac{1}{s^{1+1/\alpha}}F_{1}(s)\,\dd s.
\end{align*}
Since $F_{1}(s)$ is assumed to be bounded, we have that
\begin{align}\label{fint}
\int_{\eps_{2}}^{\infty}\frac{1}{s^{1+1/\alpha}}
F_{1}(s)\,\dd s
\le\frac{\alpha}{\eps_{2}^{1/\alpha}}\sup_{s\in(0,\infty)}F_{1}(s)
=:B_{3}<\infty.
\end{align}
Therefore, we have that
\begin{align}\label{b3}
I_{3}
\le\frac{t}{\alpha}\exp\Big[-\frac{\kappa}{2}\Big(\frac{\eps_{2}^{1/\alpha}}{t}\Big)^{{\theta}}\Big]B_{3},\quad\text{for sufficiently small }t.
\end{align}

We now work on the second integral in \eqref{integrals0}. Let $\delta>0$. It follows from \eqref{crush} and the fact that $r_{1}<\alpha$ in \eqref{r10} that we make take $\eps_{2}>0$ sufficiently small so that for all $t>0$ sufficiently small,
\begin{align}
I_{2}
& =\frac{t}{\alpha}\int_{t^{r_{1}}}^{\eps_{2}}\frac{1}{s^{1+1/\alpha}}
\exp\Big[-\kappa\Big(\frac{s^{1/\alpha}}{t}\Big)^{{\theta}}\Big]e^{g_{l}(t/s^{1/\alpha})}
e^{-C_{1}/s}e^{g_{F}(s)}\,\dd s\nonumber\\
& \le\frac{t}{\alpha}\int_{t^{r_{1}}}^{\eps_{2}}\frac{1}{s^{1+1/\alpha}}
\exp\Big[-(\kappa-\delta)\Big(\frac{s^{1/\alpha}}{t}\Big)^{{\theta}}\Big]
e^{-(C_{1}-\delta)/s}\,\dd s:=I_{2}^{+},\label{boundplus}
\end{align}
and similarly,
\begin{align}\label{boundminus}
I_{2}^{-}
:=\frac{t}{\alpha}\int_{t^{r_{1}}}^{\eps_{2}}\frac{1}{s^{1+1/\alpha}}
\exp\Big[-(\kappa+\delta)\Big(\frac{s^{1/\alpha}}{t}\Big)^{{\theta}}\Big]
e^{-(C_{1}+\delta)/s}\,\dd s
\le I_{2}.
\end{align}
Applying Lemma~\ref{exact} to $I_{2}^{\pm}$ and using the bounds \eqref{boundplus}-\eqref{boundminus} gives
\begin{align*}
&
-\frac{(C_{1}+\delta)}{\theta}(\alpha +{\theta} )  \left(\frac{\alpha  (C_{1}+\delta)}{(\kappa+\delta)  {\theta} }\right)^{-\frac{\alpha }{\alpha +{\theta} }}
=\lim_{t\to0+}t^{\alpha\theta/(\alpha+\theta)}\ln I_{2}^{-}\\
&\le \liminf_{t\to0+}t^{\alpha\theta/(\alpha+\theta)}\ln I_{2}
\le \limsup_{t\to0+}t^{\alpha\theta/(\alpha+\theta)}\ln I_{2}\\
&\le \lim_{t\to0+}t^{\alpha\theta/(\alpha+\theta)}\ln I_{2}^{+}
=-\frac{(C_{1}-\delta)}{\theta}(\alpha +{\theta} )  \left(\frac{\alpha  (C_{1}-\delta)}{(\kappa-\delta)  {\theta} }\right)^{-\frac{\alpha }{\alpha +{\theta} }}.
\end{align*}
Since $\delta>0$ is arbitrary, we have that
\begin{align*}
\lim_{t\to0+}t^{\alpha\theta/(\alpha+\theta)}\ln I_{2}
=-\frac{C_{1}}{\theta}(\alpha +{\theta} )  \left(\frac{\alpha  C_{1}}{\kappa {\theta} }\right)^{-\frac{\alpha }{\alpha +{\theta} }}.
\end{align*}
Using the bounds \eqref{b1} and \eqref{b3} and the choice of $r_{1}$ in \eqref{r10} completes the proof.
\end{proof}

\begin{proof}[Proof of Theorem~\ref{more}]

Let $\delta>0$. Again, we decompose $F_{\alpha}(t)$ into three integrals,
\begin{align}
F_{\alpha}(t)
&=\int_{0}^{t^{r_{1}}}q(s,t)F_{1}(s)\,\dd s
+\int_{t^{r_{1}}}^{\eps_{2}}q(s,t)F_{1}(s)\,\dd s
+\int_{\eps_{2}}^{\infty}q(s,t)F_{1}(s)\,\dd s\nonumber\\
&=:I_{1}+I_{2}+I_{3},\label{IS}
\end{align}
where $r_{1}>0$ satisfies \eqref{r10} and $\eps_{2}>0$ is sufficiently small so that
\begin{align}\label{e2}
\sqrt{1-\delta}
\le\frac{F_{1}(s)}{A_{1}s^{p_{1}}e^{-C_{1}/s}}
\le\sqrt{1+\delta}\quad\text{for all }s\in(0,\eps_{2}].
\end{align}
We can again choose $r_{1}$ satisfying \eqref{r10} since ${\theta}>0$ and $\alpha>0$, and we can choose $\eps_{2}$ satisfying \eqref{e2} by \eqref{asympF}.

We first bound the first integral in \eqref{IS}. Since $r_{1}>0$, the assumption in \eqref{asympF} implies that for $t$ sufficiently small,
\begin{align}\label{i1bound}
\begin{split}
I_{1}
&\le(1+\delta)\int_{0}^{t^{r_{1}}}q(s,t)A_{1}s^{p_{1}}e^{-C_{1}/s}\,\dd s\\
&\le(1+\delta)e^{-(C_{1}/2)/t^{r_{1}}}\int_{0}^{t^{r_{1}}}q(s,t)\,\dd s
\le(1+\delta)e^{-(C_{1}/2)/t^{r_{1}}}B_{1},
\end{split}
\end{align}
where $B_{1}$ is in \eqref{qint} and since $A_{1}s^{p_{1}}e^{-C_{1}/s}\le e^{-(C_{1}/2)/s}$ for all $s$ sufficiently small.

Moving to the third integral in \eqref{IS}, let $l_{0}(z):=Bz^{-{\xi}}e^{-\kappa/z^{{\theta}}}$ and notice that since $\eps_{2}>0$, \eqref{asympl} ensures that for sufficiently small $t>0$,
\begin{align}\label{i3bound}
\begin{split}
I_{3}
&\le(1+\delta)\int_{\eps_{2}}^{\infty}\frac{1}{\alpha}\frac{t}{s^{1+1/\alpha}}
B\Big(\frac{s^{1/\alpha}}{t}\Big)^{{\xi}}\exp\Big[-\kappa\Big(\frac{s^{1/\alpha}}{t}\Big)^{{\theta}}\Big]F_{1}(s)\,\dd s\\
&\le(1+\delta)\exp\Big[-\frac{\kappa}{2}\Big(\frac{\eps_{2}^{1/\alpha}}{t}\Big)^{{\theta}}\Big]t\int_{\eps_{2}}^{\infty}\frac{1}{\alpha}\frac{1}{s^{1+1/\alpha}}
BF_{1}(s)\,\dd s\\
&\le(1+\delta)\exp\Big[-\frac{\kappa}{2}\Big(\frac{\eps_{2}^{1/\alpha}}{t}\Big)^{{\theta}}\Big]tBB_{3},
\end{split}
\end{align}
where $B_{3}$ is in \eqref{fint}.

We now analyze the second integral in \eqref{IS}. Notice that
\begin{align*}
\frac{t}{\eps_{2}^{1/\alpha}}
\le \frac{t}{s^{1/\alpha}}
\le \frac{t}{t^{r_{1}/\alpha}}
= t^{1-r_{1}/\alpha}
\quad\text{for all }s\in[t^{r_{1}},\eps_{2}].
\end{align*}
Since $r_{1}<\alpha$ by \eqref{r10}, we may take $t$ sufficiently small so that
\begin{align}\label{e99}
\sqrt{1-\delta}
\le \frac{l\big(\frac{t}{s^{1/\alpha}}\big)}{l_{0}\big(\frac{t}{s^{1/\alpha}}\big)}
\le \sqrt{1+\delta}
\quad\text{for all }s\in[t^{r_{1}},\eps_{2}].
\end{align}
Therefore, for sufficiently small $t>0$, we have by \eqref{e2} and \eqref{e99} that
\begin{align*}
&I_{2}^{-}
:=(1-\delta)\int_{t^{r_{1}}}^{\eps_{2}}\frac{t}{\alpha s^{1+1/\alpha}}l_{0}\Big(\frac{t}{s^{1/\alpha}}\Big)A_{1}s^{p_{1}}e^{-C_{1}/s}(s)\,\dd s\\
&\le I_{2}
\le (1+\delta)\int_{t^{r_{1}}}^{\eps_{2}}\frac{t}{\alpha s^{1+1/\alpha}}l_{0}\Big(\frac{t}{s^{1/\alpha}}\Big)A_{1}s^{p_{1}}e^{-C_{1}/s}(s)\,\dd s=:I_{2}^{+}.
\end{align*}
Applying Lemma~\ref{exact} to $I_{2}^{\pm}$ yields
\begin{align*}
1-\delta
\le \liminf_{t\to0+}\frac{I_{2}}{At^{p}e^{-C/t^{\beta}}}
\le \limsup_{t\to0+}\frac{I_{2}}{At^{p}e^{-C/t^{\beta}}}
\le 1+\delta,
\end{align*}
where $\beta$, $C$, $A$, and $p$ are in \eqref{bcdef} and \eqref{complicated2}. Using that $\delta>0$ is arbitrary and using the bounds in \eqref{i1bound} and \eqref{i3bound} completes the proof.
\end{proof}



\begin{proof}[Proof of Corollary~\ref{corlog}]
The result follows from Theorems~\ref{log}-\ref{more} with setting $F_{1}(s)=\P(\sigma\le s)$ and $l(z)=l_{\alpha}(z)$, the asymptotic behavior in \eqref{l}, the relations in \eqref{relatime}-\eqref{taurep}, and the values of $\theta$, $\kappa$, $B$, and $\xi$ in \eqref{ldef}.
\end{proof}

\begin{proof}[Proof of Corollary~\ref{corv}]
The result follows from Theorem~\ref{log} with setting $F_{1}(s)=p_{1}(x,s\,|\,x_{0},0)$ and $l(z)=l_{\alpha}(z)$, the asymptotic behavior in \eqref{l}, Varadhan's formula in \eqref{vv}, the relation \eqref{prep}, and the values of $\theta$ and $\kappa$ in \eqref{ldef}.
\end{proof}

\begin{proof}[Proof of Corollary~\ref{corg}]
The result follows from Theorems~\ref{log}-\ref{more} with setting $F_{1}(s)=p_{1}(x,s\,|\,x_{0},0)$ and $l(z)=l_{\alpha}(z)$, the asymptotic behavior in \eqref{l}, the relation \eqref{prep}, and the values of $\theta$, $\kappa$, $B$, and $\xi$ in \eqref{ldef}.
\end{proof}


\begin{proof}[Proof of Theorem~\ref{uni}]
The result follows from Theorem~1 in \cite{lawley2020uni} and a change of variables, $t\to t^{1/\beta}$.
\end{proof}


\begin{proof}[Proof of Corollary~\ref{corcv}]
The elementary proof is similar to the proof of Corollary 2 in \cite{lawley2020mortal}. By Theorem~\ref{uni}, we have that
\begin{align}\label{varconv}
\begin{split}
(\ln N)^{2/\beta}\textup{Variance}[T_{k,N}]
&=(\ln N)^{2/\beta}\big(\E[T_{k,N}^{2}]-(\E[T_{k,N}])^{2}\big)\\
&=\frac{\E[T_{k,N}^{2}]}{(\ln N)^{-2/\beta}}-\frac{(\E[T_{k,N}])^{2}}{(\ln N)^{-2/\beta}}
\to0\quad\text{as }N\to\infty.
\end{split}
\end{align}
By Theorem~\ref{uni}, we have that for large $N$,
\begin{align}\label{ee4}
\frac{\E[T_{k,N}]}{(C/\ln N)^{1/\beta}}\ge\frac{1}{2}.
\end{align}
Let $\eps>0$. By \eqref{varconv}, we have that for large $N$,
\begin{align}\label{ee5}
(\ln N)^{1/\beta}\sqrt{\textup{Variance}[T_{k,N}]}<\eps.
\end{align}
Therefore, combining \eqref{ee4} and \eqref{ee5} gives that for large $N$,
\begin{align*}
\frac{\sqrt{\textup{Variance}[T_{k,N}]}}{\E[T_{k,N}]}
\le\frac{2\eps}{C^{1/\beta}}.
\end{align*}
Since $\eps>0$ is arbitrary, the proof is complete.
\end{proof}

\begin{proof}[Proof of Theorem~\ref{nice}]
Let $p\in(0,1)$ and observe that changing variables $z=ts^{-1/\alpha}$ yields
\begin{align}
\P(U_{\alpha}({\sigma})>t)
&=\int_{0}^{\infty}\frac{1}{\alpha}\frac{t}{s^{1+1/\alpha}}l_{\alpha}(ts^{-1/\alpha})\P({\sigma}>s)\,\dd s\nonumber\\
&=\int_{0}^{t^{p}}l_{\alpha}(z)\P({\sigma}>(t/z)^{\alpha})\,\dd z
+\int_{t^{p}}^{\infty}l_{\alpha}(z)\P({\sigma}>(t/z)^{\alpha})\,\dd z.\label{twoint}
\end{align}
Using that $\P({\sigma}>(t/z)^{\alpha})$ is an increasing function of $z$, that $l_{\alpha}(z)$ is a probability density, and the assumption in \eqref{asn}, we obtain
\begin{align}
\begin{split}\label{first09}
\int_{0}^{t^{p}}l_{\alpha}(z)\P({\sigma}>(t/z)^{\alpha})\,\dd z
\le\P({\sigma}>(t^{1-p})^{\alpha})\int_{0}^{t^{p}}l_{\alpha}(z)\,\dd z
\le\P({\sigma}>(t^{1-p})^{\alpha})\\
=\O(t^{-(1-p)\alpha r})\quad\text{as }t\to\infty.
\end{split}
\end{align}
Now, it is well-known that \cite{barkai2001}
\begin{align*}
l_{\alpha}(z)
\sim\frac{\alpha}{\Gamma(1-\alpha)}z^{-(1+\alpha)}\quad\text{as }z\to\infty.
\end{align*}
Therefore, we obtain the following bound on the asymptotic behavior of the second integral in \eqref{twoint},
\begin{align}\label{first10}
\int_{t^{p}}^{\infty}l_{\alpha}(z)\P({\sigma}>(t/z)^{\alpha})\,\dd z
\le\int_{t^{p}}^{\infty}l_{\alpha}(z)\,\dd z
=\O(t^{-p\alpha})\quad\text{as }t\to\infty.
\end{align}
Combining \eqref{twoint} with \eqref{first09} and \eqref{first10}, we obtain that there exists a constant $B_{0}$ so that
\begin{align*}
\P(U_{\alpha}({\sigma})>t)
\le B_{0}(t^{-(1-p)\alpha r}+t^{-p\alpha})\quad\text{sufficiently large }t>0.
\end{align*}
Setting $p=r/(1+r)$ yields \eqref{fr72}.

Next, if $N>\frac{1+r}{\alpha r}$, then using the definition of expectation and \eqref{fr72} yields
\begin{align*}
\E[T_{N}]
=\int_{0}^{\infty}(\P(U_{\alpha}({\sigma})>t))^{N}\,\dd t
<\infty,
\end{align*}
which yields \eqref{fr73}.
\end{proof}

\begin{proof}[Proof of Theorem~\ref{gumbel}]
The proof is similar to the proof of Proposition 3 and Theorems 1 and 2 in \cite{lawley2020dist}. Define
\begin{align}\label{s0}
S_{0}(t):=1-At^{p}e^{-C/t^{\beta}}
\end{align}
It is straightforward to check that
\begin{align*}
\lim_{t\to0+}\frac{\dd}{\dd t}\left[\frac{1-S_{0}(t)}{S_{0}'(t)}\right]
=0.
\end{align*}
Hence, Theorem 2.1.2 in \cite{falkbook} implies
\begin{align}\label{el}
\lim_{N\to\infty}(S_{0}(a_{N}x+b_{N}))^{N}
=\exp(-e^{x}),\quad\text{for all }x\in\R,
\end{align}
for some rescalings $a_{N}$ and $b_{N}$. Remark 1.1.9 in \cite{haanbook} yields that the following rescalings satisfy \eqref{el},
\begin{align}\label{ab}
a_{N}
&:=\frac{-1}{NS_{0}'(b_{N})}>0,\quad
b_{N}
:=S_{0}^{-1}(1-1/N)>0,\quad N\ge1.
\end{align}
Upon using the definition of $S_{0}$ in \eqref{s0} and properties of the LambertW function \cite{corless1996}, we obtain that the values in \eqref{ab} reduce to \eqref{abab}.

It is immediate that \eqref{el} is equivalent to
\begin{align*}
\lim_{N\to\infty}N
\ln(S_{0}(a_{N}x+b_{N}))
=-e^{x},\quad\text{for all }x\in\R.
\end{align*}
Therefore, $S_{0}(a_{N}x+b_{N})\to1$ as $N\to\infty$. Hence, L'Hospital's rule implies that
\begin{align*}
-\ln(S_{0}(a_{N}x+b_{N}))
\sim 1-S_{0}(a_{N}x+b_{N})\quad\text{as }N\to\infty.
\end{align*}
We thus conclude that \eqref{el} is equivalent to
\begin{align}\label{repl}
\lim_{N\to\infty}N(1-S_{0}(a_{N}x+b_{N}))
=e^{x},\quad\text{for all }x\in\R.
\end{align}
By assumption, $\P(\tau\le t)\sim1-S_{0}(t)$ as $t\to0+$. Therefore, \eqref{repl} holds with $S_{0}$ replaced by $S(t):=\P(\tau>t)$, and therefore \eqref{el} holds with $S_{0}$ replaced by $S$. Upon recalling the definition of convergence in distribution in \eqref{cddef00}, we conclude that the convergence in distribution in \eqref{cd} is proved for the rescalings in \eqref{abab}.

Finally, if the convergence in distribution in \eqref{cd} holds for some $\{a_{N}\}_{N\ge1}$ and $\{b_{N}\}_{N\ge1}$, then it also holds for any sequences $\{a_{N}'\}_{N\ge1}$ and $\{b_{N}'\}_{N\ge1}$ that satisfy \cite{peng2012}
\begin{align}\label{cc1}
\lim_{N\to\infty}\frac{a_{N}'}{a_{N}}
=1,\quad
\lim_{N\to\infty}\frac{b_{N}'-b_{N}}{a_{N}}
=0.
\end{align}
Hence, basic properties of the LambertW function \cite{corless1996} yield \eqref{ababuf}.
\end{proof}

\begin{proof}[Proof of Theorem~\ref{kth}]
The result follows immediately from Theorem~\ref{gumbel} above and Theorem 3.5 in \cite{colesbook}.
\end{proof}

\begin{proof}[Proof of Theorem~\ref{moments}]
The proof is similar to the proofs of Theorems~3 and 5 in \cite{lawley2020dist}.
\end{proof}

\subsection{Short-time behavior of drift-diffusion}

For the drift-diffusion process in Section~\ref{pearson}, it is well-known \cite{feller1968} that the probability density of ${\sigma}$ is $f(s):=\tfrac{\dd}{\dd s}\P(\sigma\le s)=f_{0}(s)+f_{1}(s)$, where
\begin{align}\label{f0}
f_{0}(s)
=e^{-vy}e^{-(v^{2}/2)\overline{s}}\frac{1}{\sqrt{2\pi \overline{s}^{3}}}\sum_{k=-\infty}^{\infty}(y+2k)\exp\Big(-\frac{2}{\overline{s}}(y/2+k)^{2}\Big),
\end{align}
with $y=x_{0}/{L_{0}}$, $\overline{s}=(2K_{\alpha}/{L_{0}}^{2})s$, and $v=({L_{0}}/(2K_{\alpha})V_{\alpha}$, and the formula for $f_{1}$ is obtained from \eqref{f0} and replacing $v$ by $-v$ and $y$ by $1-y$. Therefore, $f(s)$ has the short-time behavior,
\begin{align}\label{fas}
f(s)
\sim
\begin{cases}
e^{-vy}\frac{y}{\sqrt{2\pi \overline{s}^{3}}}\exp (-\frac{y^{2}}{2\overline{s}} ) & \text{if }y<1/2,\\
e^{v(1-y)}\frac{1-y}{\sqrt{2\pi \overline{s}^{3}}}\exp (-\frac{(1-y)^{2}}{2\overline{s}} ) & \text{if }y>1/2,\\
(e^{-vy}+e^{v(1-y)})\frac{y}{\sqrt{2\pi \overline{s}^{3}}}\exp (-\frac{y^{2}}{2\overline{s}} ) & \text{if }y=1/2.
\end{cases}
\end{align}
Taking the Laplace transform of \eqref{fas}, dividing by the Laplace transform variable, and then taking the inverse Laplace transform yields
\begin{align*}
\P({\sigma}\le s)
\sim A s^{1/2}\exp\Big(\frac{d_{0}^{2}}{4K_{\alpha}s}\Big)\quad\text{as }s\to0+,
\end{align*}
where $d_{0}:=\min\{x_{0},{L_{0}}-x_{0}\}>0$ and 
\begin{align*}
A
=\begin{cases}
\sqrt{\frac{2K_{\alpha}}{{L_{0}}^{2}}}\frac{2{L_{0}}^{2}}{d_{0}^{2}}
\exp(-\frac{V_{\alpha}}{2K_{\alpha}}x_{0})\frac{x_{0}/{L_{0}}}{\sqrt{2\pi }} & \text{if }x_{0}<{L_{0}}/2,\\
\sqrt{\frac{2K_{\alpha}}{{L_{0}}^{2}}}\frac{2{L_{0}}^{2}}{d_{0}^{2}}
\exp(\frac{V_{\alpha}}{2K_{\alpha}}({L_{0}}-x_{0}))\frac{1-x_{0}/{L_{0}}}{\sqrt{2\pi }} & \text{if }x_{0}>{L_{0}}/2,\\
\sqrt{\frac{2K_{\alpha}}{{L_{0}}^{2}}}\frac{2{L_{0}}^{2}}{d_{0}^{2}}
(\exp(-\frac{V_{\alpha}}{2K_{\alpha}}x_{0})+\exp(\frac{V_{\alpha}}{2K_{\alpha}}({L_{0}}-x_{0})))\frac{x_{0}/{L_{0}}}{\sqrt{2\pi }} & \text{if }x_{0}={L_{0}}/2.
\end{cases}
\end{align*}


\bibliography{library.bib}
\bibliographystyle{unsrt}

\end{document}